\documentclass [reqno] {amsart}

\usepackage{amsfonts}
\usepackage{amssymb}

\newtheorem{thm}{Theorem}[section]
\newtheorem{cor}[thm]{Corollary}
\newtheorem{lem}[thm]{Lemma}

\theoremstyle{definition}
\newtheorem{defn}[thm]{Definition} %[section]
\newtheorem{exmp}[thm]{Example} %[section]

\theoremstyle{remark}
\newtheorem{rem}[thm]{Remark} %[section]

\numberwithin{equation}{section}

\newcommand{\cA}{\mathcal{A}}
\newcommand{\C}{\mathbb{C}}

\newcommand{\cD}{\mathcal{D}}
\newcommand{\iD}{\mathit{\Delta}}
\newcommand{\Id}{I}
\newcommand{\SL}{\mathit{SL}}
\newcommand{\R}{\mathbb{R}}
\newcommand{\gS}{\mathfrak{S}}
\newcommand{\cX}{\mathcal{X}}

\begin{document}

\title[Root Functions of Sturm-Liouville Problems]
      {On Completeness of Root Functions
\\
       of Sturm-Liouville  Problems
\\
       with Discontinuous Boundary Operators\footnote{This is a preprint version 
			of the paper published in Journal of Differential 
Equations, 255 (2013), 3305-3337}}

%% First author

\author{A. Shlapunov}

\address[Alexander Shlapunov]
        {Siberian Federal University
\\
         Institute of Mathematics
\\
         pr. Svobodnyi 79
\\
         660041 Krasnoyarsk
\\
         Russia}

\email{ashlapunov@sfu-kras.ru}

%% Second author

\author{N. Tarkhanov}

\address[Nikolai Tarkhanov]
        {Universit\"{a}t Potsdam
\\
         Institut f\"{u}r Mathematik
\\
         Am Neuen Palais 10
\\
         14469 Potsdam
\\
         Germany}

\email{tarkhanov@math.uni-potsdam.de}

\date{November 14, 2011}



\begin{abstract}
We consider a Sturm--Liouville boundary value problem in a boun\-ded domain $\cD$ of
$\mathbb{R}^n$. By this is meant that the differential equation is given by a second order
elliptic operator of divergent form in $\cD$ and the boundary conditions are of Robin type on
$\partial \cD$. The first order term of the boundary operator is the oblique derivative whose
coefficients bear discontinuities of the first kind.
Applying the method of weak perturbation of compact self-adjoint operators and the method
of rays of minimal growth, we prove the completeness of root functions related to the boundary
value problem in Lebesgue and Sobolev spaces of various types.
\end{abstract}

\maketitle

%\vspace*{2.5cm}

\tableofcontents

%\newpage

\section*{Introduction}
\label{s.Int}

The Hilbert space methods take considerable part in the modern theory of partial differential
equations.
In particular, the spectral theorem for compact self-adjoint operators attributed to Hilbert and
Schmidt allows one to look for solutions of boundary value problems for formally self-adjoint
operators in the form of expansions over eigenfunctions of the operator.

Non-self-adjoint compact operators fail to have eigenvectors in general.
Keldysh \cite{Keld51} (see also \cite{Keld71} and
                                \cite[Ch. 5, \S 8]{GokhKrei69} for more details)
elaborated expansions over root functions for weak perturbations of compact self-adjoint operators.
In particular, he applied successfully the theorem on the completeness of root functions to the
Dirichlet problem for second order elliptic operators in divergent form.

The problem of completeness of the system of eigen- and associated functions of boundary value
problems for elliptic operators in domains with smooth boundary was studied in many articles
(see for instance \cite{Brow53},
                  \cite{Brow59a},
                  \cite{Brow59b},
                  \cite{Agmo62},
                  \cite{Kond99}).
In a series of papers \cite{Agra94a},
                      \cite{Agra94b},
                      \cite{Agra08},
                      \cite{Agra11b},
                      \cite{Agra11c},
including two serveys \cite{Agra02} and
                      \cite{Agra11a},
Agranovich proved the completeness of root functions for a wide class of boundary value problems
for second order elliptic equations with boundary conditions of the Dirichlet,
                                                                Neumann and
                                                                Zaremba
type in standard Sobolev spaces over domains with Lipschitz boundary.
Note that the class of Lipschitz surfaces does not include surfaces with general conical points,
as they are introduced in analysis on singular spaces.

Root functions of general elliptic boundary value problems in weighted Sobolev spaces over domains
with conic and edge type singularities on the boundary were studied in \cite{EgorKondSchu01} and
\cite{Tark06}.
These papers used estimates of the resolvent of compact operators and the so-called rays of
minimal growth. In order to realize fully to what extent the completeness criteria of
\cite{EgorKondSchu01} and
\cite{Tark06}
are efficient, we dwell on the concept of ellipticity on a compact manifold with smooth edges on
the boundary. Such a singular space $\cX$ has three smooth strata, more precisely,
the interior part $\cX_0$ of $\cX$,  the smooth part $\cX_1$ of the boundary
and the edge $\cX_2$ which is assumed to be a compact closed manifold.
Pseudodifferential operators on $\cX$ are $(3 \times 3)\,$-matrices $\cA$ whose entries $A_{i,j}$
are operators mapping functions on $\cX_j$ to functions on $\cX_i$.
To each operator $\cA$ one assigns a principal symbol
   $\sigma (\cA) := (\sigma_0 (\cA),\sigma_1 (\cA),\sigma_2 (\cA))$
in such a way that
   $\sigma (\cA) = 0$ if and only if $\cA$ is compact,
and
   $\sigma (\mathcal{B} \cA) = \sigma (\mathcal{B})\, \sigma (\cA)$
for all operators $\cA$ and $\mathcal{B}$ whose composition is well defined.
The components $\sigma_i (\cA)$ of the principal symbol are functions on the cotangent bundles of
$\cX_i$ with values in operator spaces.
They are smooth away from zero sections of the bundles and bear certain twisted homogeneity as
operator families.
An operator $\cA$ is called elliptic if its principal symbol is invertible away from the zero
sections of cotangent bundles.
The invertibility of $\sigma_0 (\cA)$ just amounts to the ellipticity of $\cA$ in the interior
of $\cX$. The invertibility of $\sigma_1 (\cA)$ is equivalent to the Shapiro-Lopatinskii condition
on the smooth part of $\partial \cX$.
The invertibility of $\sigma_2 (\cA)$ constitutes the most difficult problem, for this operator
family is considered in weighted Sobolev spaces on an infinite cone.
An operator $\cA$ proves to be Fredholm if and only if it is elliptic.
However, from what has been said it follows that there is no efficient criteria of ellipticity on
compact manifolds with edges on the boundary.
In general these techniques allow one to derive at most the following result.
Consider a classical boundary value problem on $\cX$ satisfying the Shapiro-Lopatinskii condition
away from the edge $\cX_2$.
It is actually given by a column of operators $A_{i,0}$ with $i = 0, 1$, where
   $A_{0,0}$ is an elliptic differential operator in $\cX_0$
and
   $A_{1,0}$ a differential operator near $\cX_1$ followed by restriction to $\cX_1$.
We complete the column to a $(2 \times 2)\,$-matrix $A$ by setting $A_{0,1} = 0$ and
                                                                   $A_{1,1} = 0$.
The Shapiro-Lopatinskii condition implies that
   $\sigma_2 (A) (y,\eta)$
is a family of Fredholm operators on the unit sphere in $T^\ast \cX_2$.
Hence we can set $\sigma_2 (A) (y,\eta)$ in the frame of a $(3 \times 3)\,$-matrix $a (y,\eta)$ on
the unit sphere of $T^\ast \cX_2$ which is moreover invertible.
A distinct quantisation procedure leads then immediately to a Fredholm operator of the type
\begin{equation}
\label{eq.Fo}
   \left( \begin{array}{cc}
          A_{0,0} & A_{0,2}
\\
          A_{1,0} & A_{1,2}
\\
          A_{2,0} & A_{2,2}
          \end{array}
   \right) :\,
   \begin{array}{c}
   C^{\infty} (\cX)
\\
   \oplus
\\
   C^{\infty} (\cX_2, \mathbb{C}^{l_1})
   \end{array}
 \to
   \begin{array}{c}
   C^{\infty} (\cX)
\\
   \oplus
\\
   C^{\infty} (\partial \cX, \mathbb{C}^{m})
\\
   \oplus
\\
   C^{\infty} (\cX_2, \mathbb{C}^{l_2})
   \end{array},
\end{equation}
where $l_1$ and $l_2$ are non-negative integers.
However, the Fredholm property of (\ref{eq.Fo}) elucidates by no means the original problem
$$
   \left\{ \begin{array}{rclcl}
             A_{0,0} u
           & =
           & f
           & \mathrm{in}
           & \cX_0,
\\
             A_{1,0} u
           & =
           & u_0
           & \mathrm{on}
           & \cX_1,
           \end{array}
   \right.
$$
unless $\cX_2$ is of dimension $0$.
Thus, operator-valued symbols make the condition of ellipticity ineffective.

In the present paper we study the completeness of root functions for the Sturm-Liouville boundary
value problems for second order elliptic operators in divergent form with Robin type boundary
conditions containing oblique derivative with discontinuous coefficients.
The discontinuities are of first kind along a smooth hypersurface on the boundary of $\cD$.
This hypersurface can be thought of as edge on the boundary, hence we are within the framework of
analysis on compact manifolds with edges of codimension $1$ on the boundary.
The theory of \cite{Tark06} applies in this situation provided that one is able to establish the
invertibility of the edge symbol.
This latter is a family of Sturm-Liouville boundary value problems in an infinite plane cone
parametrized by the points of the cotangent bundle of the edge.
On reducing the family to the boundary of the cone one obtains two ordinary pseudodifferential
equations on the rays constituting the boundary of the cone.
The invertibility of the edge symbol just amounts to the unique solvability of these equations in
certain weighted Sobolev spaces on the rays.
Making this the subject of our next paper, we restrict our attention here on variational solutions
of the Sturm-Liouville boundary value problem.
In the absence of regularity theorem we are forced to apply the weak perturbation theory, thus
showing the completeness of root functions in the $L^2$ space.

\section{Weak perturbations of compact self-adjoint operators}
\label{s.2}

Let $H$ be a separable (complex) Hilbert space and $A: H \to H$ a linear operator.
As usual, $\lambda \in \C$ is said to be an eigenvalue of $A$ if there is a non-zero element
$u \in H$, such that $(A - \lambda I)  u = 0$, where $\Id$ is the identity operator in $H$.
The element $u$ is called an eigenvector of $A$ corresponding to the eigenvalue $\lambda$.
When supplemented with the zero element, all eigenvectors corresponding to an eigenvalue $\lambda$
form a vector subspace $E (\lambda)$ in $H$.
It is called an eigenspace of $A$ corresponding to $\lambda$, and the dimension of $E (\lambda)$
is called the (geometric) multiplicity of $\lambda$.
The famous spectral theorem of Hilbert and Schmidt asserts that the system of eigenvectors of a
compact self-adjoint operator in $H$ is complete.

\begin{thm}
\label{t.spectral}
Let
   $A: H \to H$ be compact and self-adjoint.
Then
   all eigenvalues of $A$ are real,
   each non-zero eigenvalue has finite multiplicity,
and
   the system of all eigenvalues counted with their multiplicities is countable and has the only
   accumulation point $\lambda = 0$.
Moreover, there is an orthonormal basis in $H$ consisting of eigenvectors of $A$.
\end{thm}

As already mentioned, a non-self-adjoint compact operator might have no eigenvalues.
However, each non-zero eigenvalue (if exists) is of finite multiplicity, see for instance
   \cite{DunfSchw63}.
Similarly to the Jordan normal form of a linear operator on a finite-dimensional vector space one
uses the more general concept of root functions of operators.

More precisely, an element $u \in H$ is called a root vector of $A$ corresponding to an eigenvalue
$\lambda \in \C$ if
$
(A - \lambda \Id)^m u = 0
$
for some natural number $m$.
The set of all root vectors corresponding to an eigenvalue $\lambda$ form a vector subspace in $H$
whose dimension is called the (algebraic) multiplicity of $\lambda$.

If the linear span of the set of all root elements is dense in $H$ one says that the root elements
of $A$ are complete in $H$.
Aside from self-adjoint operators, the question arises under what conditions on a compact operator
$A$ the system of its root elements is complete.

If the dimension of $H$ is finite then the completeness is equivalent to the possibility of
reducing the matrix $A$ to the Jordan normal form.
Of course, this is always the case for linear operators in complex vector spaces, see, for
instance,
   \cite[\S~88]{Waer67}.

In order to formulate the simplest completeness result for Hilbert spaces we need the definition
of the order of a compact operator $A$.
Since $A: H \to H$ is compact, the operator $A^\ast A$ is compact, self-adjoint and non-negative.
Hence it follows that $A^\ast A$  possesses a unique non-negative self-adjoint compact square root
   $(A^\ast A)^{1/2}$
often denoted by $|A|$.
By Theorem \ref{t.spectral} the operator $|A|$ has countable system of non-negative eigenvalues
   $s_\nu (A)$
which are called the $s\,$-numbers of $A$.
It is clear that if $A$ is self-adjoint then $s_\nu = |\lambda_\nu|$, where $\{ \lambda_\nu \}$ is
the system of eigenvalues of $A$.

\begin{defn}
\label{d.Schatten}
The operator $A$ is said to belong to the Schatten class $\gS_p$, with $0 < p < \infty$,
if
$$
   \sum_\nu |s_\nu (A)|^p < \infty.
$$
\end{defn}

Note that $\gS_2$ is the set of all Hilbert-Schmidt operators while
          $\gS_1$ is the ideal of all trace class operators.

The following lemma will be very useful in the sequel; it is taken from \cite{DunfSchw63}
   (see also \cite[Ch.~2, \S~2]{GokhKrei69}).

\begin{lem}
\label{l.2}
Let $A$ be a compact operator of class $\gS_p$, with $0 < p < \infty$, in a Hilbert space $H$, and
$B$ be a bounded operator in $H$.
Then the compositions $B A$ and
                      $A B$
belong to $\gS_p$.
\end{lem}

After M.V. Keldysh a compact operator $A$ is said to be of finite order if it belongs to a
Schatten class $\gS_p$.
The infinum of such numbers $p$ is called the order of $A$.
The following result is usually referred to as theorem on weak perturbations of compact
self-adjoint operators.
It was first proved in \cite{Keld51}, see also \cite{Keld71}.
Here we present its formulation from \cite[Ch.~5, \S~8]{GokhKrei69}.

\begin{thm}
\label{t.keldysh}
Let $A_0$ be a compact self-adjoint operator of finite order in $H$.
If
   $\delta A$ is a compact operator
and
   the operator $A_0 (\Id + \delta A)$ is injective,
then
   the system of root elements of $A_0 (\Id + \delta A)$ is complete in $H$
and,
   for any $\varepsilon > 0$, all eigenvalues of $A_0 (\Id + \delta A)$ (except for a finite number)
   belong to the angles $|\arg \lambda| < \varepsilon$ and
                  $|\arg \lambda - \pi| < \varepsilon$.
Moreover,

1)
If $A_0$ has only a finite number of negative eigenvalues, then $A_0 (\Id + \delta A)$ has at most a
finite number of  eigenvalues in the angle $|\arg \lambda - \pi| < \varepsilon$.

2)
If $A_0$ has only a finite number of positive eigenvalues, then $A_0 (\Id + \delta A)$ has at most a
finite number of  eigenvalues in the angle $|\arg \lambda| < \varepsilon$.
\end{thm}

As is easy to see, both operators $A_0 (\Id + \delta A)$ and
                                  $A_0$
are in fact injective under the hypothesis of Theorem \ref{t.keldysh}.

\section{The Sturm-Liouville problem}
\label{s.4}

By a Sturm-Liouville problem in $\R^n$ we mean any boundary value problem for solutions of second
order elliptic partial differential equation with Robin-type boundary condition.
The coefficients of the Robin boundary condition are allowed to have discontinuities of the first
kind, and so mixed boundary conditions are included as well.

Let $\cD$ be a bounded domain in $\R^n$ with Lipschitz boundary, i.e., the surface $\partial \cD$
is locally the graph of a Lipschitz function.
In particular, the boundary $\partial \cD$ possesses a tangent hyperplane almost everywhere.

We consider complex-valued functions defined in the domain $\cD$ and its closure $\overline{\cD}$.
For $1 \leq q \leq \infty$, we write $L^q (\cD)$ for the space of all measurable
functions $u$ in $\cD$,
such that the integral of $|u|^q$ over $\cD$ is finite.
For functions $u \in C^s (\overline{\cD})$ we introduce the norm
$$
   \| u \|_{H^s (\cD)}
 = \Big( \int_{\cD} \sum_{|\alpha| \leq s} |\partial^\alpha u|^2\, dx \Big)^{1/2},
$$
where
$
   \partial^\alpha
 = \partial_1^{\alpha_1} \cdots \partial_n^{\alpha_n}
$
for a multi-index $\alpha = (\alpha_1, \ldots \alpha_n)$, and
   $\partial_j = \partial / \partial x_j$.
The completion of the space $C^s (\overline{\cD})$ with respect to this norm is the Banach space
   $H^s (\cD)$.
It is convenient to set $H^0 (\cD) := L^2 (\cD)$.
Then $H^s (\cD)$ is a Hilbert space with scalar product
$$
   ( u,v)_{H^s (\cD)}
 = \int_{\cD} \sum_{|\alpha| \leq s} \partial^\alpha u \overline{\partial^\alpha v}\, dx,
$$
$u, v \in H^s (\cD)$.

Consider a second order partial differential operator $A$ in the domain $\cD$ of divergence form
$$
   A (x,\partial) u
 = - \sum_{i, j = 1}^n \partial_i ( a_{i,j} (x) \partial _j u)
   + \sum_{j = 1}^n a_j (x) \partial_j u
   + a_0 (x) u,
$$
the coefficients $a_{i,j}$, $a_j$ and $a_0$ being of class $L^{\infty} (\cD)$.

Let
   $v (x) = (v_1 (x), \ldots, v_n (x))$
be a vector field on the surface $\partial \cD$.
Denote by $\partial_v$ the oblique derivative
$$
   \partial_v = \sum_{j=1}^n v_j (x) \partial_j
$$
and introduce a first order boundary operator
$
   B (x,\partial)
 = \partial_v + b_0 (x).
$
The coefficients $v_1 (x), \ldots, v_n (x)$ and
                 $b_0 (x)$
are assumed to be bounded measurable functions on $\partial \cD$.
We allow the vector $v (x)$ to vanish on an open connected subset $S$ of $\partial \cD$
with piecewise smooth boundary $\partial S$.
In this case we assume that $b_0 (x)$ does not vanish on $S$.

Consider the following boundary value problem with Robin-type condition on the surface
$\partial \cD$.
Given a distribution $f$ in $\cD$, find a distribution $u$ in $\cD$ which satisfies
\begin{equation}
\label{eq.SL}
   \left\{
\begin{array}{rclcl}
     A u
   & =
   & f
   & \mbox{in}
   & \cD,
\\
     B u
   & =
   & 0
   & \mbox{on}
   & \partial \cD.
\end{array}
   \right.
\end{equation}
In order to get asymptotic results, it is necessary to put some restrictions on the operators
$A$ and
$B$.
Suppose that the matrix
$$
   \left( a_{i,j} (x) \right)_{\substack{i = 1, \ldots, n \\
                                         j = 1, \ldots, n}}
$$
is Hermitian and
\begin{equation}
\label{eq.ell}
   \sum_{i,j=1}^n a_{i,j} (x) \xi_i \xi_j
 \geq
   m\, |\xi|^2
\end{equation}
for all
   $(x,\xi) \in \cD \times (\R^n \setminus \{ 0 \})$,
with $m$ a positive constant independent on $x$ and $\xi$.
Estimate (\ref{eq.ell}) is nothing but the statement that the operator $A$ is strongly elliptic.
Note that inequality (\ref{eq.ell}) is weaker than the coercivity, i.e., the existence of a
constant $m$ such that
\begin{equation}
\label{eq.coercive}
   \sum_{i,j=1}^n a_{i,j} (x)\, \overline{\partial_i u} \partial_j u
 \geq
   m \sum_{j=1}^n |\partial_j u|^2
\end{equation}
for all $u \in C^1 (\overline{\cD})$, because $u$ may take on complex values.

On assuming that $a_{i,j} (x)$ are continuous up to the boundary of $\cD$ we consider the complex
vector field $c (x)$ on $\partial \cD$ with components
$$
   c_j (x) = \sum_{i=1}^n a_{i,j} (x) \nu_i (x),
$$
where
   $\nu (x) = (\nu_1 (x), \ldots, \nu_n (x))$
is the unit outward normal vector of $\partial \cD$ at $x \in \partial \cD$.
From condition (\ref{eq.ell}) it follows that there is a complex-valued function $b_1 (x)$ on the
boundary with the property that the difference
   $v (x) - b_1 (x) c (x)$
is orthogonal to $\nu (x)$ for almost all $x \in \partial \cD$.
In fact, the pointwise equality $(v - b_1 c, \nu) = 0$ just amounts to
$$
   b_1 (x)
 = \frac{\displaystyle \sum_{j=1}^{n} v_j (x) \nu_j (x)}
        {\displaystyle \sum_{i,j=1}^{n} a_{i,j} (x) \nu_i (x) \nu_j (x)}
$$
for $x \in \partial \cD$.
Obviously, $b_1 (x)$ is a bounded measurable function on $\partial \cD$ and the vector field
   $t (x) = v (x) - b_1 (x) c (x)$
takes on its values in the tangent hyperplane
   $T_x (\partial \cD)$
of $\partial \cD$ at $x$.
Summarizing we conclude that if the boundary of $\cD$ is a Lipschitz surface then
\begin{equation}
\label{eq.B}
   B (x,\partial) = b_1 (x) \partial_c + \partial_t + b_0 (x),
\end{equation}
where
   $b_1 \in L^\infty (\partial \cD)$
and
   $t (x)$ is a tangential vector field on $\partial \cD$ whose components belong to
   $L^\infty (\partial \cD)$.
By assumption, both $b_1$ and $t$ vanish on $S$.
Concerning the behavior of $b_1$ in $\partial \cD \setminus S$ we require that
   $b_1 (x) \neq 0$ for almost all $x \in \partial \cD \setminus S$
and
   $1 / b_1 \in L^1 (\partial \cD \setminus S)$.

Note that the Shapiro-Lopatinskii condition is violated on $\partial \cD \setminus S$
unless the coefficients $a_{i,j} (x)$ are real-valued.

Denote by $H^1 (\cD,S)$ the subspace of $H^1 (\cD)$ consisting of those functions whose
restriction to the boundary vanishes on $S$.
This space is Hilbert under the induced norm.
It is easily seen that smooth functions on $\overline{D}$ vanishing in a neighbourhood of
$\overline{S}$ are dense in $H^1 (\cD,S)$.
Since on $S$ the boundary operator reduces to $B = b_0$ and $b_0 (x) \neq 0$ for $x \in S$,
the functions of $H^1 (\cD)$ satisfying $Bu = 0$ on $\partial \cD$ belong to $H^1 (\cD,S)$.

As we want to study perturbation of self-adjoint operators we split both $a_0$ and $b_0$ into
two parts
$$
\begin{array}{rcl}
   a_0
 & =
 & a_{0,0} + \delta a_0,
\\
   b_0
 & =
 & b_{0,0} + \delta b_0,
\end{array}
$$
where
   $a_{0,0}$ is a non-negative bounded function in $\cD$
and
   $b_{0,0}$ a bounded function on $\partial \cD$ satisfying $b_{0,0} / b_1 \geq 0$.
If the functional
$$
   \| u \|_{\SL}
 = \Big(
   \int_{\cD} \sum_{i,j=1}^n a_{i,j}  \overline{\partial_i u} \partial_j u\, dx
 + \| \sqrt{a_{0,0}} u \|_{L^2 (\cD)}^2
 + \| \sqrt{b_{0,0} b_1^{-1}} u \|_{L^2 (\partial \cD \setminus S)}^2
   \Big)^{1/2}
$$
defines a norm on $H^1 (\cD,S)$, we denote by $H_{\SL} (\cD)$ the completion of $H^1 (\cD,S)$
with respect to this norm.
Then $H_{\SL} (\cD)$ is actually a Hilbert space with scalar product
$$
   (u,v)_{\SL}
 =  \int_{\cD} \sum_{i,j=1}^n a_{i,j} \partial_j u \overline{\partial_i v}\, dx
 + (a_{0,0} u, v)_{L^2 (\cD)}
 + (b_{0,0} b_1^{-1} u, v)_{L^2 (\partial \cD \setminus S)}.
$$

From now on we assume that the space $H_{\SL} (\cD)$ is continuously embedded into the Lebesgue
space $L^2 (\cD)$, i.e.,
\begin{equation}
\label{eq.incl.est}
   \| u \|_{L^2 (\cD)}
 \leq c\, \| u \|_{\SL}
\end{equation}
for all $u \in H_{\SL} (\cD)$, where $c$ is a constant independent of $u$.
This condition is not particularly restrictive, as examples below show.
If the coefficients $a_{i,j}$ are smooth up to $S$ then the functions of $H_{\SL} (\cD)$
belong actually to $H^1_{\mathrm{loc}} (\cD \cup S)$, for the Dirichlet problem for strongly
elliptic equations satisfies the Shapiro-Lopatinskii condition.

Write $\iota$ for the inclusion
\begin{equation}
\label{eq.incl}
   H_{\SL} (\cD) \hookrightarrow L^2 (\cD),
\end{equation}
which is continuous by (\ref{eq.incl.est}).
We use this inclusion to specify the dual space of $H_{\SL} (\cD)$ via the pairing in $L^2 (\cD)$.
More precisely, let $H_{\SL}^{-} (\cD)$ be the completion of $H^1 (\cD,S)$ with respect to the
norm
$$
   \| u \|_{H^-_{\SL} (\cD)}
 = \sup_{\substack{v \in H^1 (\cD,S) \\ v \ne 0}}
   \frac{|(v,u)_{L^2 (\cD)}|}{\| v \|_{\SL}}.
$$

\begin{rem}
\label{r.norms-}
As $H^1 (\cD,S)$ is dense in
   $H_{\SL} (\cD)$
and the norm $\|\cdot\|_{SL}$ majorizes $\| \cdot \|_{L^2 (\cD)}$ we see that
$$
   \| u \|_{H^-_{\SL} (\cD)}
 = \sup_{\substack{v \in H_{SL} (\cD) \\ v \ne 0}}
   \frac{|(v,u)_{L^2 (\cD)}|}{\| v \|_{\SL}}.
$$
\end{rem}

\begin{lem}
\label{l.dual.emb}
The space $L^2 (\cD)$ is continuously embedded into $H^-_{\SL} (\cD)$.
If inclusion (\ref{eq.incl}) is compact then the space $L^2 (\cD)$ is compactly embedded into
                                                       $H^-_{\SL} (\cD)$.
\end{lem}

\begin{proof}
By definition and estimate (\ref{eq.incl.est}) we get
$$
   \| u \|_{H^-_{\SL} (\cD)}
 \leq  \sup_{\substack{v \in H^1 (\cD,S) \\ v \ne 0}}
       \frac{\| u \|_{L^2 (\cD)} \| v \|_{L^2 (\cD)}}{\| v \|_{\SL}}
 \leq c\, \| u \|_{L^2 (\cD)}
$$
for all $u \in L^2 (\cD)$,
   i.e., the space $L^2 (\cD)$ is continuously embedded into $H^-_{\SL} (\cD)$ indeed.

Suppose (\ref{eq.incl}) is compact.
Then the Hilbert space adjoint
$
   \iota^\ast :  L^2 (\cD) \hookrightarrow H_{\SL} (\cD)
$
is compact, too.
By Remark \ref{r.norms-}) we conclude that
\begin{eqnarray} \label{eq.iota*}
   \| u \|_{H^-_{\SL} (\cD)}
 & = &
   \sup_{\substack{v \in H^1 (\cD,S) \\ v \ne 0}}
   \frac{|(\iota (v),u)_{L^2 (\cD)}|}{\| v \|_{\SL}}
\nonumber
\\
 & = &
   \sup_{\substack{v \in H_{SL} (\cD) \\ v \ne 0}}
   \frac{|(v, \iota^\ast (u))_{\SL}|}{\| v \|_{\SL}}
\nonumber
\\
 & = &
   \| \iota^\ast (u) \|_{\SL}
\end{eqnarray}
for all $u \in L^2 (\cD)$.
Therefore, any weakly convergent sequence in $L^2 (\cD)$ converges in
                                             $H^-_{\SL} (\cD)$,
which shows the second part of the lemma.
\end{proof}

Since
   $C^\infty_{\mathrm{comp}} (\cD)$ is dense in $L^2 (\cD)$ and
   the norm $\| \cdot \|_{L^2 (\cD)}$ majorizes the norm
            $\| \cdot \|_{H^-_{\SL} (\cD)}$,
we conclude that $C^\infty_{\mathrm{comp}} (\cD)$ is dense in $H^-_{\SL} (\cD)$, too.

\begin{lem}
\label{l.dual}
The Banach space $H^-_{\SL} (\cD)$ is topologically isomorphic to the dual space
$H_{\SL} (\cD)'$ and the isomorphism is defined by the sesquilinear form
\begin{equation}
\label{eq.pair}
   (v,u) = \lim_{\nu \to \infty} (v,u_\nu)_{L^2 (\cD)}
\end{equation}
for $u \in H_{\SL}^- (\cD)$ and
    $v \in H_{\SL} (\cD)$
where $\{ u_\nu \}$ is any sequence in $H^1 (\cD,S)$ converging to $u$.
\end{lem}

That is,
   for every fixed $u \in H_{\SL}^- (\cD)$, pairing (\ref{eq.pair}) defines a
   continuous linear functional $f_u$ on $H_{\SL} (\cD)$
and,
   for each $f \in H_{\SL} (\cD)'$, there is a unique $u \in H_{\SL}^- (\cD)$ with
   $f (v)= f_u (v)$ for all $v \in H_{\SL} (\cD)$.
Moreover, the conjugate linear map $u \mapsto f_u$ is an isometry.

\begin{proof}
To show that the limit on the right-hand side of (\ref{eq.pair}) exists for each
fixed function $v \in H_{\SL} (\cD)$, is suffices to show that
   $\{ (v,u_\nu)_{L^2 (\cD)} \}$
is a Cauchy sequence.
By definition,
$$
   | (v,u_\nu - u_\mu)_{L^2 (\cD)} |
 \leq
   \| v \|_{SL} \| u_\nu - u_\mu \|_{H_{\SL}^- (\cD)}
 \to 0
$$
as $\nu, \mu \to \infty$, which is our claim.
Clearly, this limit does not depend on the particular sequence $\{ u_\nu \}$, for if
   $\| u_\nu \|_{H_{\SL}^- (\cD)} \to 0$,
then
   $| (v,u_\nu)_{L^2 (\cD)} | \to 0$
for all $v \in H_{\SL} (\cD)$.

From the definition it follows that
\begin{equation}
\label{eq.pair.est}
   |(v,u)|
 \leq
   \| u \|_{H_{\SL}^- (\cD)}\, \| v \|_{SL}
\end{equation}
for all $u \in H_{\SL}^- (\cD)$ and
        $v \in H_{\SL} (\cD)$.
Hence, for every fixed $u \in H_{\SL}^- (\cD)$, the formula
   $f_u (v) := (v,u)$
defines a continuous linear functional $f_u$ on $H_{\SL} (\cD)$, such that
$$
   \| f_u \|_{H_{\SL} (\cD)'} \leq \| u \|_{H_{\SL}^- (\cD)}.
$$

If $\{ u_\nu \} \subset H^1 (\cD,S)$ approximates an element $u$ in $H_{\SL}^- (\cD)$,
then equality (\ref{eq.iota*}) implies that the sequence
   $\{ \iota^\ast \iota\, u_\nu \}$
converges to a function $U$ in the space $H_{\SL} (\cD)$ and
\begin{eqnarray*}
   \| U \|_{SL}
 & = &
   \lim_{\nu \to \infty} \| \iota^\ast \iota\, u_\nu \|_{SL}
\\
 & = &
   \lim_{\nu \to \infty} \| u_\nu \|_{H^-_{\SL} (\cD)}
\\
 & = &
   \| u \|_{H^-_{\SL} (\cD)}.
\end{eqnarray*}
Moreover,
\begin{eqnarray*}
   f_u (v)
 & = &
   (v,u)
\\
 & = &
   \lim_{\nu \to \infty } (\iota\, v, \iota\, u_\nu)_{L^2 (\cD)}
\\
 & = &
   \lim_{\nu \to \infty } (v, \iota^\ast \iota\, u_\nu)_{SL}
\\
 & = &
   (v, U) _{SL}
 \end{eqnarray*}
for all $v \in H_{\SL} (\cD)$.
Now, the Riesz theorem yields
   $\| U \|_{SL} = \| f_u \|_{H_{\SL} (\cD)'}$
whence
   $\| f_u \|_{H_{\SL} (\cD)'} = \| u \|_{H^-_{\SL} (\cD)}$.

It remains to show that any continuous linear functional $f$ on $H_{\SL} (\cD)$ has the form
$f (v) = (v, u_f)$ for some $u_f \in H_{\SL}^- (\cD)$.
According to the Riesz theorem, for any $f \in H_{\SL} (\cD)'$ there is a unique element
   $U_f \in H_{\SL} (\cD)$,
such that
$$
   f (v) = (v, U_f)_{SL}
$$
for all $v \in  H_{\SL} (\cD)$.
Besides, $\| U_f \|_{SL} = \| f \|_{H_{\SL} (\cD)'}$.
By definition, the operator $\iota$ is injective and its image is dense in $L^2 (\cD)$.
Hence, the image of the operator $\iota^\ast \iota$ in $H_{\SL} (\cD)$ is dense, too.
Pick a sequence $\{ u_\nu  \} \subset H_{\SL} (\cD)$ with the property that
   $ \{ \iota^\ast \iota\, u_\nu \}$
converges to $U_f$.
Then, according to (\ref{eq.iota*}), $\{ u_\nu \}$ is a Cauchy sequence in $H^-_{\SL} (\cD)$,
and so it converges to an element $u_f$ in this space.
It as easy to see that $u_f$ is actually independent of the particular choice of the sequence
   $\{ u_\nu \}$.
Finally, we obtain
\begin{eqnarray*}
   (v,u_f)
 & = &
   \lim_{\nu \to \infty} (\iota\, v, \iota\, u_\nu)_{L^2 (\cD)}
\\
 & = &
   \lim_{\nu \to \infty} (v, \iota^\ast \iota\, u_\nu)_{SL}
\\
 & = &
   (v, U_f)_{SL}
 \\
 & = &
   f (v)
\end{eqnarray*}
for all $v \in H_{SL} (\cD)$, as desired.
\end{proof}

Note that $H_{\SL} (\cD)$ is reflexive, since it is a Hilbert space.
Hence it follows that
   $(H_{\SL}^- (\cD))' = H_{\SL} (\cD)$,
i.e., the spaces $H_{\SL} (\cD)$ and
                 $H_{\SL}^- (\cD)$
are dual to each other with respect to (\ref{eq.pair}).

Given any $u \in H^2 (\cD)$ and
          $v \in H^1 (\cD)$,
the Stokes formula yields
$$
   \int_{\partial \cD} \overline{v} \partial_c u\, ds
 = \int_{\cD}
   \sum_{i,j=1}^n \left( a_{i,j} \overline{\partial_i v} \partial_j u
                       + \overline{v} \partial_i (a_{i,j} \partial_j u \right) dx.
$$
On integrating by parts we see that
\begin{eqnarray*}
\lefteqn{
   (Au, v)_{L^2 (\cD)}
}
\\
 \!\! & \!\! = \!\! & \!\!
   \int_{\cD} \sum_{i,j=1}^n \! a_{i,j} \overline{\partial_i v} \partial_j u\, dx
 + \left( b_1^{-1} (\partial_t + b_0) u, v \right)_{L^2 (\partial \cD \setminus S)}
 + \Big( \sum_{j=1}^n \! a_{j} \partial_j u + a_0 u, v \Big)_{L^2 (\cD)}
\end{eqnarray*}
for all
   $u \in H^2 (\cD)$ and
   $v \in H^1 (\cD)$
satisfying the boundary condition of (\ref{eq.SL}).
Suppose
\begin{equation}
\label{eq.positive.part1}
   \Big|
   \left( b_1^{-1} (\partial_t + \delta b_0) u, v \right)_{L^2 (\partial \cD \setminus S)}
 + \Big( \sum_{j=1}^n \! a_{j} \partial_j u + \delta a_0\, u, v \Big)_{L^2 (\cD)}
   \Big|
 \leq
   c\, \| u \|_{\SL} \| v \|_{\SL}
\end{equation}
for all $u, v \in  H^1 (\cD,S)$, where $c$ is a positive constant independent of $u$ and $v$.
Then, for each fixed $u \in H_{\SL} (\cD)$, the sesquilinear form
\begin{eqnarray*}
\lefteqn{
   Q (u, v)
}
\\
 \!\! & \!\! = \!\! & \!\!
   \int_{\cD} \sum_{i,j=1}^n \! a_{i,j} \overline{\partial_i v} \partial_j u\, dx
 + \left( b_1^{-1} (\partial_t + b_0) u, v \right)_{L^2 (\partial \cD \setminus S)}
 + \Big( \sum_{j=1}^n \! a_{j} \partial_j u + a_0 u, v \Big)_{L^2 (\cD)}
\end{eqnarray*}
determines a continuous linear functional $f$ on $H_{\SL} (\cD)$ by
   $f (v) := \overline{Q (u,v)}$
for $v \in H_{\SL} (\cD)$.
By Lemma \ref{l.dual}, there is a unique element in $H_{\SL}^- (\cD)$,
   which we denote by $Lu$,
such that
$$
   f (v) = (v,Lu)
$$
for all $v \in H_{\SL} (\cD)$.
We have thus defined a linear operator
   $L : H_{\SL} (\cD) \to H_{\SL}^- (\cD)$.
From (\ref{eq.positive.part1}) it follows that $L$ is bounded.

The bounded linear operator $L_0 : H_{\SL} (\cD) \to H_{\SL}^- (\cD)$ defined in the
same way via the sesquilinear form $(\cdot,\cdot)_{\SL}$, i.e.,
\begin{equation}
\label{eq.L0}
   (v,u)_{\SL} = (v, L_0 u)
\end{equation}
for all $u, v \in H_{\SL} (\cD)$, corresponds to the case
   $a_j \equiv 0$ for all $j = 1, \ldots, n$,
   $a_0 = a_{0,0}$,
and
   $t \equiv 0$,
   $b_0 = b_{0,0}$.

We are thus lead to a weak formulation of problem (\ref{eq.SL}).
Given $f \in H_{\SL}^- (\cD)$, find $u \in H_{\SL} (\cD)$, such that
\begin{equation}
\label{eq.SL.w}
   \overline{Q (u,v)} = (v,f)
\end{equation}
for all $v \in H_{\SL} (\cD)$.

It should be noted that there is no need to assume the continuity of the coefficients
$a_{i,j}$ up to the boundary of $D$ in order to define the operator $L$ and to consider
the weak formulation of problem (\ref{eq.SL}).
Now one can handle problem (\ref{eq.SL.w}) by standard techniques of functional analysis,
   see for instance \cite[Ch.~3, \S \S~4-6]{LadyUral73}) for the coercive case.

Suppose $f \in L^2 (\cD)$.
If $u$ is a solution of (\ref{eq.SL.w}) then $Au = f$ holds in the sense of distributions
in $\cD$, for $C^\infty_{\mathrm{comp}} (\cD) \subset H_{\SL} (\cD)$.
Since $A$ is elliptic, we readily deduce that $u \in H^2_{\mathrm{loc}} (\cD)$ and the
equality $Au = f$ is actually satisfied almost everywhere in $\cD$.

If, in addition, $u \in H^2 (\cD)$ then
$$
   \left( (\partial_c + b_1^{-1} (\partial_t + b_0)) u, v \right)_{L^2 (\partial \cD \setminus S)}
 = 0
$$
for all $v \in H_{\SL} (\cD)$.
Since any smooth function $v$ in $\overline{\cD}$ whose support does not meet $S$ belongs to
$H_{\SL} (\cD)$, we conclude that
   $(b_1 \partial_c + \partial_t + b_0) u = 0$
on $\partial \cD \setminus S$.
Hence $Bu = 0$ on $\partial \cD$, for $u = 0$ and $b_1 = 0$ on $S$.

\begin{lem}
\label{l.solv.SL1}
Let
   estimate (\ref{eq.incl.est}) hold true.
Assume that
   $a_j \equiv 0$ for all $j = 1, \ldots, n$,
   $\delta a_0 = 0$,
and
   $t \equiv 0$,
   $\delta b_0 = 0$.
Then for each $f \in H^-_{\SL} (\cD)$ there is a unique solution $u \in H_{\SL} (\cD)$ to
problem (\ref{eq.SL.w}), i.e., the operator $L_0 : H_{\SL} (\cD) \to H_{\SL}^- (\cD)$ is
continuously invertible.
Moreover, the norms of both $L_0$ and its inverse $L_0^{-1}$ are equal to $1$.
\end{lem}

\begin{proof}
Under the hypotheses of the lemma, (\ref{eq.SL.w}) is just a weak formulation of problem
(\ref{eq.SL}) with $A$ and $B$ replaced by
\begin{eqnarray*}
   A_0
 & = &
 - \sum_{i,j = 1}^n \partial_i (a_{i,j} \partial_j) + a_{0,0},
\\
   B_0
 & = &
   b_1 \partial_c + b_{0,0},
\end{eqnarray*}
respectively.
The corresponding bounded operator in Hilbert spaces just amounts to
   $L_0 : H_{\SL} (\cD) \to H_{\SL}^- (\cD)$
defined by (\ref{eq.L0}).
Its norm equals $1$, for, by Lemma \ref{l.dual}, we get
\begin{equation}
\label{eq.unit.norm}
   \| L_0 u \|_{H_{\SL}^- (\cD)}
 =
   \sup_{\substack{v \in H_{\SL} (\cD) \\ v \ne 0}}
   \frac{|(v, L_0 u)|}{\| v \|_{\SL}}
 =
   \sup_{\substack{v \in H_{\SL} (\cD) \\ v \ne 0}}
   \frac{|(v,u)_{\SL}|}{\| v \|_{\SL}}
 =
   \| u \|_{\SL}
\end{equation}
whenever $u \in H_{\SL} (\cD)$.

The existence and uniqueness of solutions to problem (\ref{eq.SL.w}) follows
immediately from the Riesz theorem on the general form of continuous linear functionals
on Hilbert spaces.
From (\ref{eq.unit.norm}) we conclude that $L_0$ is actually an isometry of
   $H^-_{\SL} (\cD)$ onto
   $H_{\SL} (\cD)$,
as desired.
\end{proof}

Consider the sesquilinear form on $H^-_{\SL} (\cD)$ given by
$$
   (u,v)_{H^-_{\SL} (\cD)} := (L_0^{-1} u, v)
$$
for $u, v \in H_{\SL}^- (\cD)$.
Since
\begin{equation}
\label{eq.L0.L0-}
   (L_0^{-1} u, v)
 = (L^{-1}_0 u, L_0 L^{-1}_0  v)
 = (L_0^{-1} u, L_0^{-1} v)_{\SL}
\end{equation}
for all $u, v \in H^-_{\SL} (\cD)$, the last equality being due to (\ref{eq.L0}), this form is
Hermitian.
Combining (\ref{eq.unit.norm}) and
          (\ref{eq.L0.L0-})
yields
$$
   \sqrt{(u, u)_{H^-_{\SL} (\cD)}} = \| u \|_{H^-_{\SL} (\cD)}
$$
for all $u \in H^-_{\SL} (\cD)$.
From now on we endow the space $H^-_{\SL} (\cD)$ with the scalar product
   $(\cdot,\cdot)_{H^-_{\SL} (\cD)}$.

\begin{lem}
\label{l.solv.SL2}
Let estimates (\ref{eq.incl.est}) and
              (\ref{eq.positive.part1})
be fulfilled.
If moreover the constant $c$ of (\ref{eq.positive.part1}) satisfies $c < 1$ then, for each
$f \in H^-_{\SL} (\cD)$, there exists a unique solution $u \in H_{\SL} (\cD)$ to problem
(\ref{eq.SL.w}), i.e., the operator $L : H_{\SL} (\cD) \to H^-_{\SL} (\cD)$ is contin\-uously
invertible.
\end{lem}

\begin{proof}
If estimate (\ref{eq.positive.part1}) holds with $c < 1$ then the operator
   $L : H_{\SL} (\cD) \to H^-_{\SL} (\cD)$
corresponding to problem (\ref{eq.SL.w}) is easily seen to differ from $L_0$ by a bounded
operator
   $\delta L : H_{\SL} (\cD) \to H^-_{\SL} (\cD)$
whose norm does not exceed $c < 1$.
As $L_0$ is invertible and the inverse operator $L_0^{-1}$ has norm $1$, a familiar argument
shows that $L$ is invertible, too.
\end{proof}

Since
   $C^\infty_{\mathrm{comp}} (\cD) \hookrightarrow H_{\SL} (\cD) \hookrightarrow L^2 (\cD)$,
the elements of $H^-_{\SL} (\cD)$ are distributions in $\cD$ and any solution to problem
(\ref{eq.SL}) satisfies $Au = f$ in $\cD$ in the sense of distributions.
Though the boundary conditions are interpreted in a weak sense, they agree with those in
terms of restrictions to $\partial \cD$ if the solution is sufficiently smooth up to the
boundary, e.g. belongs to $C^1 (\overline{\cD})$.
In particular, if $L$ is invertible then the inverse operator maps $H^-_{\SL} (\cD)$ to
                                                                   $H_{\SL} (\cD)$
where elements $u$ satisfy $Bu = 0$ on $\partial \cD$ in a suitable sense.

\section{Completeness of root functions for weak perturbations}
\label{s.root}

We are now in a position to study the completeness of root functions related to problem
(\ref{eq.SL.w}). We begin with the self-adjoint operator $L_0$.
To this end we write $\iota'$ for the continuous embedding of $L^2 (\cD)$ into
                                                              $H^-_{\SL} (\cD)$,
as it is described by Lemma \ref{l.dual.emb}.

\begin{lem}
\label{l.L0.selfadj}
Suppose that estimate (\ref{eq.incl.est}) is fulfilled and
             inclusion (\ref{eq.incl}) is compact.
Then the inverse $L^{-1}_0$ of the operator given by (\ref{eq.L0}) induces compact positive
self-adjoint operators
$$
\begin{array}{rclcrcl}
   Q_{1}
 & =
 & \iota' \iota\, L^{-1}_{0}
 & :
 & H^-_{\SL} (\cD)
 & \to
 & H^-_{\SL} (\cD),
\\
   Q_{2}
 & =
 & \iota\, L^{-1}_{0}\, \iota'
 & :
 & L^2 (\cD)
 & \to
 & L^2 (\cD),
\\
   Q_{3}
 & =
 & L^{-1}_{0}\, \iota' \iota
 & :
 & H_{\SL} (\cD)
 & \to
 & H_{\SL} (\cD)
\end{array}
$$
which have the same systems of eigenvalues and eigenvectors.
Moreover, all eigenvalues are positive and there are orthonormal bases in
   $H_{\SL} (\cD)$,
   $L^2 (\cD)$ and
   $H^-_{\SL} (\cD)$
consisting of the eigenvectors.
\end{lem}

\begin{proof}
Since $\iota$ is compact and
      $\iota'$,
      $L^{-1}_{0}$
bounded, all the operators $Q_1$,
                           $Q_2$,
                           $Q_3$
are compact.

Recall that we endow the space $H^-_{\SL} (\cD)$ with the scalar product
   $(\cdot,\cdot)_{H^-_{\SL} (\cD)}$.
Then, by (\ref{eq.L0.L0-}),
\begin{eqnarray*}
   (Q_{1} u, v)_{H_{\SL}^- (\cD)}
 & = &
   \overline{(v, \iota' \iota\, L^{-1}_{0} u)_{H_{\SL}^- (\cD)}}
\\
 & = &
   \overline{(L^{-1}_{0} v, \iota' \iota\, L^{-1}_{0} u)}
\\
 & = &
   (\iota L^{-1}_{0} u , \iota L^{-1}_{0} v)_{L^2 (\cD)},
\end{eqnarray*}
and
\begin{eqnarray*}
   (u, Q_{1} v)_{H_{\SL}^- (\cD)}
 & = &
   \overline{(Q_{1} v, u)_{H_{\SL}^- (\cD)}}
\\
 & = &
   (\iota L^{-1}_{0} u, \iota L^{-1}_{0} v)_{L^2 (\cD)}
\end{eqnarray*}
for all $u,v \in H^-_{\SL} (\cD)$, i.e., the operator $Q_{1}$ is self-adjoint.

Using (\ref{eq.L0}) we get
\begin{eqnarray*}
   (Q_{2} u, v)_{L^2 (\cD)}
 & = &
   (\iota (L^{-1}_{0} (\iota' u)), v)_{L^2 (\cD)}
\\
 & = &
   (L^{-1}_{0} (\iota' u), \iota' v)
\\
 & = &
   (L^{-1}_{0} (\iota' u), L^{-1}_{0} (\iota' v))_{\SL}
\end{eqnarray*}
and
\begin{eqnarray*}
   (u, Q_{2} v)_{L^2 (\cD)}
 & = &
   \overline{(Q_{2} v, u)_{L^2 (\cD)}}
\\
 & = &
   (L^{-1}_{0} (\iota' u), L^{-1}_{0} (\iota' v))_{\SL}
\end{eqnarray*}
for all $u, v \in L^2 (\cD)$, i.e., the operator $Q_{2}$ is self-adjoint.

On applying (\ref{eq.L0}) once again we obtain
\begin{eqnarray*}
   (Q_{3} u, v)_{\SL}
 & = &
   (L^{-1}_{0} (\iota' \iota\, u), v)_{\SL}
\\
 & = &
   (\iota' \iota\, u, v)
\\
 & = &
   (\iota u, \iota v)_{L^2 (\cD)}
\end{eqnarray*}
and
\begin{eqnarray*}
   (u, Q_{3} v)_{\SL}
 & = &
   \overline{(Q_{3} v, u)_{\SL}}
\\
 & = &
   (\iota u, \iota v)_{L^2 (\cD)}
\end{eqnarray*}
for all $u, v \in H_{\SL} (\cD)$, which establishes the self-adjointness of $Q_{3}$.

Finally, as the operator $L_0^{-1}$ is injective, so are the operators $Q_1$,
                                                                       $Q_2$ and
                                                                       $Q_3$.
Hence, all their eigenvectors $\{ u_\nu \}$ belong to the space $H_{\SL} (\cD)$,
for $L_0^{-1} u_\nu$ lies in $H_{\SL} (\cD)$.
From the injectivity of $\iota$ and $\iota'$ we conclude that the systems of eigenvalues
and eigenvectors of $Q_{1}$, $Q_{2}$ and $Q_{3}$ coincide.
The last part of the lemma follows from Theorem \ref{t.spectral}.
\end{proof}

Our next goal is to apply Theorem \ref{t.keldysh} for studying the completeness of root functions
of weak perturbations of $Q_j$.

\begin{thm}
\label{p.root.func.1}
If the operator
   $Q_{1} : H^-_{\SL} (\cD) \to H^-_{\SL} (\cD)$
is of finite order then,
   for any invertible operator of the type
$
   L_0 + \delta L : H_{\SL} (\cD) \to H^-_{\SL} (\cD)
$
with a compact operator
$
   \delta L : H_{\SL} (\cD) \to H^-_{\SL} (\cD),
$
the system of root functions of the compact operator
$$
   P_{1} = \iota' \iota\, (L_0 + \delta L)^{-1} :  H^-_{\SL} (\cD) \to H^-_{\SL} (\cD)
$$
is complete in the spaces $H^-_{\SL} (\cD)$,
                          $L^2 (\cD)$ and
                          $H_{\SL} (\cD)$.
\end{thm}

\begin{proof}
By assumption there is a bounded inverse
   $(L_0 + \delta L)^{-1} : H^-_{\SL} (\cD) \to H_{\SL} (\cD)$.
Since
$$
   \Id - L_0 (L_0 + \delta L)^{-1} = \delta L\, (L_0 + \delta L)^{-1},
$$
we conclude that
\begin{equation}
\label{eq.perturbation}
\begin{array}{rcl}
   L_{0}^{-1} - (L_0 + \delta L) ^{-1}
 & =
 & L_0^{-1} \left( \delta L\, (L_0 + \delta L)^{-1} \right),
\\
   Q_{1} - P_{1}
 & =
 & Q_{1} \left( \delta L\, (L_0 + \delta L)^{-1} \right).
\end{array}
\end{equation}
From the
   compactness of $\delta L$ and
   boundedness of $(L_0 + \delta L)^{-1}$
it follows that the operator
$$
   \delta L\, (L_0 + \delta L)^{-1} : H^-_{\SL} (\cD) \to H^-_{\SL} (\cD)
$$
is compact.

Hence, $P_{1}$ is an injective weak perturbation of the compact self-adjoint operator
$Q_{1}$.
If in addition the order of $Q_{1}$ is finite then Theorem \ref{t.keldysh} implies that
the countable system  $\{ u_\nu \}$ of root functions related to the operator $P_{1}$
is complete in the Hilbert space $H^-_{\SL} (\cD)$.

Pick a root function $u_\nu$ of the operator $P_{1}$ corresponding to an eigenvalue
$\lambda_\nu$.
Note that $\lambda_\nu \ne 0$, for the operator $(L_0 + \delta L)^{-1}$ is injective.
By definition there is a natural number $m$, such that
$
   (P_{1} - \lambda_\nu \Id)^m u_\nu = 0.
$
Using the binomial formula yields
$$
   \sum_{j=0}^m \Big( \substack{m \\ j} \Big) \lambda_\nu^{m-j} P_{1}^j u_\nu = 0.
$$
In particular, since $\lambda_\nu \ne 0$, we get
$$
   u_\nu
 = \sum_{j=1}^m
   \Big( \substack{m \\ j} \Big) \lambda_\nu^{-j} (\iota' \iota\, (L_0 + \delta L)^{-1})^j u_{\nu}.
$$
Hence, $u_\nu \in H_{\SL} (\cD)$ because the range of the operator $(L_0 + \delta L)^{-1}$
lies in the space $H_{\SL} (\cD)$.

We have thus proved that $\{ u_\nu \} \subset H_{\SL} (\cD)$.
Our next concern will be to show that the linear span
   ${\mathcal L} (\{ u_\nu \})$
of the system $\{ u_\nu \}$ is dense in $H_{\SL} (\cD)$
   (cf. Proposition 6.1 of \cite{Agra11a} and \cite[p.~12]{Agra11c}).
For this purpose, pick $u \in H_{\SL} (\cD)$.
As $L_0 + \delta L$ maps $H_{\SL} (\cD)$ continuously onto
                      $H^-_{\SL} (\cD)$,
we get $(L_0 + \delta L) u \in H^-_{\SL} (\cD)$.
Hence, there is a sequence $\{ f_k \} \subset {\mathcal L} (\{ u_\nu \})$ converging to
$(L_0 + \delta L) u$ in $H^-_{\SL} (\cD)$.
On the other hand, the inverse $(L_0 + \delta L)^{-1}$ maps $H^-_{\SL} (\cD)$ continuously to
                                                            $H_{\SL} (\cD)$,
and so the sequence
$$
   (L_0 + \delta L)^{-1} f_k = (L_0 + \delta L)^{-1} \iota' \iota\, f_k
$$
converges to $u$ in $H_{\SL} (\cD)$.

If now $u_{\nu_0} \in {\mathcal L} (\{ u_\nu \})$ corresponds to an eigenvalue $\lambda_0$
of multiplicity $m_0$ then the vector $v_{\nu_{0}} = P_1  u_{\nu_0}$ satisfies
$$
   (P_1 - \lambda_0 \Id)^{m_0} v_{\nu_{0}}
 = (P_1 - \lambda_0 \Id)^{m_0+1} u_{\nu_{0}} + \lambda_0 (P_1 - \lambda_0 \Id)^{m_0} u_{\nu_{0}}
 = 0.
$$
Thus, the operator $P_1$ maps ${\mathcal L} (\{ u_\nu \})$ to
                              ${\mathcal L} (\{ u_\nu \})$
itself.
Therefore, the sequence
   $\{ \iota' \iota\, (L_0 + \delta L)^{-1} f_k \}$
still belongs to ${\mathcal L} (\{ u_\nu \})$ and we can think of
   $\{ (L_0 + \delta L)^{-1} f_k \} $
as sequence of linear combinations of root functions of $P_1$ converging to $u$.
These arguments show that the subsystem
$
   (L_0 + \delta L)^{-1}\, {\mathcal L} (\{ u_\nu \}) \subset {\mathcal L} (\{ u_\nu \})
$
is dense in $H_{\SL} (\cD)$.

Finally, since
   $C^\infty_{\mathrm{comp}} (\cD) \subset H_{\SL} (\cD)$
and
   $C^\infty_{\mathrm{comp}} (\cD)$ is dense in the Lebesgue space $L^2 (\cD)$,
the space $H_{\SL} (\cD)$ is dense in $L^2 (\cD)$ as well.
This proves the completeness of the system of root functions in $L^2 (\cD)$.
\end{proof}

Similar assertions are also true for the weak perturbations of the operators $Q_2$ and
                                                                             $Q_3$.

The operator
   $L_0 + \delta L : H_{\SL} (\cD) \to H^-_{\SL} (\cD)$
with a compact operator $\delta L$ fails to be injective in general, and so
Theorem \ref{p.root.func.1} does not apply.
However, as $L_0$ is continuously invertible, we conclude that $L = L_0 + \delta L$ is
Fredholm.
In particular, there is a constant $c$, such that
\begin{equation}
\label{eq.Fredholm}
   \| u \|_{\SL}
 \leq
   c\, \left( \| Lu \|_{H_{\SL}^- (\cD)} + \| u \|_{H_{\SL}^- (\cD)} \right)
\end{equation}
for all $u \in H_{\SL} (\cD)$.

We next extend Theorem \ref{p.root.func.1} to Fredholm operators.
To this end denote by $T$ the unbounded linear operator
   $H^-_{\SL} (\cD) \to H^-_{\SL} (\cD)$
with domain $\cD_{T} = H_{\SL} (\cD)$ which maps an element $u \in \cD_{T}$ to $Lu$.
The operator $T$ is clearly closed because of inequality (\ref{eq.Fredholm}).
It is densely defined as $H^1 (\cD,S) \subset H_{\SL} (\cD)$ is dense in $H^-_{\SL} (\cD)$.
It is well known that the null space of $T$ is finite dimensional in $H_{\SL} (\cD)$ and
                      its range is closed in $H^-_{\SL} (\cD)$.

When speaking on eigen- and root functions $u$ of the operator $T$ we always assume that
$u \in \cD_{T}$ and
   $(T - \lambda \Id)^{j} u \in \cD_{T}$
for all $j = 1, \ldots, m-1$.

Let $T_0: H^-_{\SL} (\cD) \to H^-_{\SL} (\cD)$ correspond to the self-adjoint operator
$L_0$.
The operator $T_0$ is obviously continuously invertible and the inverse operator coincides
with $\iota' \iota\, L_0^{-1} = Q_1$.

\begin{lem}
\label{l.spectrum.T0}
The spectrum of the operator $T_0$ consists of the points
   $\mu_\nu = \lambda_\nu^{-1}$
in ${\mathbb R}_{> 0}$, where $\lambda_\nu$ are the eigenvalues of $Q_1$.
\end{lem}

\begin{proof}
Recall that all $\lambda_\nu$ are positive, which is due to Lemma \ref{l.L0.selfadj}, and so
$\mu_\nu > 0$.
If $\lambda \ne 0$ then
$$
   (T_0 - \lambda \Id) u
 = (\Id  - \lambda\, \iota' \iota\, L^{-1}_0)\, T_0 u
 = - \lambda \Big( Q_1 - \frac{1}{\lambda} \Id \Big)\, T_0 u
$$
for all $u \in H_{\SL} (\cD)$, showing the lemma.
\end{proof}

If the spectrum of $T$ fails to be the whole complex plane, i.e., if the resolvent
$
   \mathcal{R} (\lambda; T)
 = (T - \lambda \Id)^{-1}
$
exists for some $\lambda = \lambda_0$, then it follows from the resolvent equation
   (since $\mathcal{R} (\lambda_0; T)$ is compact)
that $\mathcal{R} (\lambda; T)$ exists for all $\lambda \in \C$ except for a discrete
sequence of points $\{ \lambda_{\nu} \}$ which are the eigenvalues of $T$
   (see \cite[p.~17]{Keld71}.
In the general case, however, one cannot exclude the situation where the spectrum of $T$
is the whole complex plane.

\begin{thm}
\label{t.root.func.1}
Let
   $\delta L : H_{\SL} (\cD) \to H^-_{\SL} (\cD)$ be a compact operator and
   the operator $Q_{1}: H^-_{\SL} (\cD) \to H^-_{\SL} (\cD)$ be of finite order.
Then the spectrum of the closed operator
$
   T : H^-_{\SL} (\cD) \to H^-_{\SL} (\cD)
$
corresponding to
$
   L = L_0 + \delta L,
$
is different from $\C$ and the system of root functions of $T$ is complete in the spaces
   $H^-_{\SL} (\cD)$,
   $L^2 (\cD)$ and
   $H_{\SL} (\cD)$.
Moreover, for any $\varepsilon > 0$, all eigenvalues of $T$ (except for a finite number)
belong to the corner $|\arg \lambda| < \varepsilon$.
\end{thm}

\begin{proof}
First we note that
\begin{equation}
\label{eq.T-mu}
   T - \lambda \Id
 = L - \lambda\, \iota' \iota
\end{equation}
on $H_{\SL} (\mathcal{D})$ for all $\lambda \in \C$.
Let us prove that there is $N \in \mathbb{N}$, such that
   $\lambda_0 = - N$
is a resolvent point of $T$.
For this purpose, using (\ref{eq.T-mu}) and Lemma \ref{l.spectrum.T0}, we get
\begin{equation}
\label{eq.spectrum.T}
   T + k \Id
 = (\Id + \delta L (L_0 + k \, \iota' \iota)^{-1}) (T_0 + k \Id)
\end{equation}
for all $k \in \mathbb{N}$.

We will show that the operator
   $\Id + \delta L (L_0 + k\, \iota' \iota)^{-1}$
is injective for some $k \in \mathbb{N}$.
Indeed, we argue by contradiction.
Suppose for any $k \in \mathbb{N}$ there is $f_k \in H^-_{\SL} (\cD)$, such that
   $\| f_k \|_{H_{\SL}^- (\cD)} = 1$
and
\begin{equation}
\label{eq.cee}
   (\Id + \delta L (L_0 + k\, \iota' \iota)^{-1}) f_k = 0.
\end{equation}
Given any $u \in H_{\SL} (\cD)$ and
          $k \in \mathbb{N}$,
an easy computation shows that
\begin{eqnarray*}
   \|(L_0 + k\, \iota' \iota) u \|_{H_{\SL}^- (\cD)}^2
 & = &
   \| u + k\, L_0^{-1} u \|_{\SL}^2
\\
 & = &
   \| u \|_{\SL}^2 + 2k\, \| u \|_{L^2 (\cD)}^2 + k^2\, \| L_0^{-1} u \|_{\SL}^2
\\
 & \geq &
   \| u \|_{\SL}^2.
\end{eqnarray*}
Hence, the sequence
   $u_k := (L_0 + k\, \iota' \iota)^{-1} f_k$
is bounded in $H_{\SL} (\cD)$.
Now the weak compactness principle for Hilbert spaces yields that there is a subsequence
   $\{ f_{k_j} \}$
with the property that both
   $\{ f_{k_j} \}$ and
   $\{ u_{k_j} \}$
converge weakly in the spaces $H^-_{\SL} (\cD)$ and
                              $H_{\SL} (\cD)$
to limits $f$ and $u$, respectively.
Since $\delta L$ is compact, it follows that the sequence $\{ \delta L u_{k_j} \}$
converges to $\delta L u$ in $H^-_{\SL} (\cD)$, and so $\{ f_{k_j} \}$ converges to $f$
because of (\ref{eq.cee}).
Obviously,
$$
   \| f \|_{H_{\SL}^- (\cD)} = 1.
$$
In particular, we conclude that the sequence
   $\{ \delta L (L_0 + k_j\, \iota' \iota)^{-1} f_{k_j} \}$
converges to $- f$ whence
\begin{equation}
\label{eq.w00}
   f = - \delta L\, u.
\end{equation}

Further, on passing to the weak limit in the equality
   $f_{k_j} = (L_0 + k_j\, \iota' \circ \iota) u_{k_j}$
we obtain
$$
   f = L_0 u - \lim_{k_j \to \infty} k_j\, \iota' \iota\,  u_{k_j},
$$
for the continuous operator
   $L_0 : H_{\SL} (\cD) \to H^-_{\SL} (\cD)$
maps weakly convergent sequences to weakly convergent sequences.
As the operator $\iota' \iota$ is compact, the sequence
   $\{ \iota' \iota\, u_{k_j} \}$
converges to $\iota' \iota\,  u$ in the space $H^-_{\SL} (\cD)$  and
   $\iota' \iota\, u \ne 0$
which is a consequence of (\ref{eq.w00}) and the injectivity of $\iota' \iota$.
This shows readily that the weak limit
$$
   \lim_{k_j \to \infty} k_j\, \iota' \iota\, u_{k_j}
 = L_0 u - f
$$
does not exist, a contradiction.

We have proved more, namely that the operator
   $\Id + \delta L (L_0 + k\, \iota' \iota)^{-1}$
is injective for all but a finitely many natural numbers $k$.
Since this is a Fredholm operator of index zero, it is continuously invertible.
Hence, (\ref{eq.spectrum.T}) and
       Lemma \ref{l.spectrum.T0}
imply that $(T - \lambda_0 \Id)^{-1}$ exists for some $\lambda_0 = - N$ with
$N \in \mathbb{N}$.

As $\lambda_0$ is a resolvent point of $T$,
$$
   (T - \lambda_0 \Id)^{-1}
 = (L - \lambda_0\, \iota' \iota )^{-1}
$$
on $H^-_{\SL} (\cD)$.
Since
   $L : H_{\SL} (\cD) \to H^-_{\SL} (\cD)$ is Fredholm and
   the inclusion $\iota$ compact,
the operator
   $L - \lambda_0\, \iota' \iota : H_{\SL} (\cD) \to H^-_{\SL} (\cD)$
is Fredholm.
So $(L - \lambda_0\, \iota' \iota)^{-1}$ maps $H^-_{\SL} (\cD)$ continuouosly to
                                              $H_{\SL} (\cD)$.
Similarly to (\ref{eq.perturbation}) we get
$$
   L_{0}^{-1} - (L - \lambda_0\, \iota' \iota)^{-1}
 = L_0^{-1}
   \left( (\delta L  - \lambda_0\, \iota' \iota) (L - \lambda_0\, \iota' \iota)^{-1}
   \right).
$$
Then, Theorem \ref{p.root.func.1} yields that the root functions $\{ u_\nu \}$ of the
operator
   $(L - \lambda_0\, \iota' \iota) ^{-1}$
are complete in the spaces $H_{\SL} (\cD)$,
                           $L^2 (\cD)$ and
                           $H^-_{\SL} (\cD)$.

From (\ref{eq.T-mu}) it follows that the systems of root functions related to the operators
   $(L - \lambda_0\, \iota' \iota)^{-1}$ and
   $T - \lambda_0\, \Id$
coincide.

Finally, as the operators $T - \lambda_0 \Id$ and
                          $T$
have the same root functions, we conclude that $\mathcal{L} (\{ u_\nu \})$ is dense in the
spaces
   $H_{\SL} (\cD)$,
   $L^2 (\cD)$ and
   $H^-_{\SL} (\cD)$.
\end{proof}

The equality $(T - \lambda \Id) u = 0$ for a function $u \in H_{\SL} (\cD)$ may be equivalently reformulated by saying that $u$ is a solution in a weak sense to the boundary value problem
\begin{equation}
\label{eq.SLep}
\left\{
\begin{array}{rclcl}
   A u
 & =
 & \lambda u
 & \mbox{in}
 & \cD,
\\
   B u
 & =
 & 0
 & \mbox{on}
 & \partial \cD,
\end{array}
\right.
\end{equation}
where the pair $(A,B)$ corresponds to the perturbation $L_0 + \delta L$.
For $n = 1$ such problems are known as Sturm-Liouville boundary problems for second order
ordinary differential equations (see for instance \cite[Ch.~XI, \S~4]{Hart64}).
Thus, we may still refer to (\ref{eq.SLep}) as the Sturm-Liouville problem in many dimensions.

To finish the preliminary discussion we indicate a broad class of compact operators of
finite orders acting in spaces of integrable functions.
For this purpose, given any positive non-integer $s$, we denote by $H^{s} (\cD)$ the
Sobolev-Slobodetskii space with norm
$$
   \| u \|_{H^s (\cD)}
 = \Big( \| u \|^2_{H^{[s]} (\cD)}
       + \int \!\!\! \int_{\cD \times \cD}
         \sum_{|\alpha | = [s]}
         \frac{|\partial^\alpha u (x) - \partial^\alpha u (y)|^2}{|x - y|^{n+2(s-[s])}}\,
         dx dy
   \Big)^{1/2},
$$
where $[s]$ is the integer part of $s$.
For $s < 0$, the space $H^{s} (\cD)$ is defined to be the dual space of $H^{-s} (\cD)$
with respect to the $L^2 (\cD)\,$-pairing, as discussed in Section \ref{s.4}.
More precisely, by $H^{s} (\cD)$ is meant the completion of $H^{-s} (\cD)$ with respect
to the norm
$$
   \| u \|_{H^{s} (\cD)}
 = \sup_{\substack{v \in H^{-s} (\cD) \\ v \ne 0}}
   \frac{|(v,u)_{L^2 (\cD)}|}{\| v \|_{H^{-s} (\cD)}}.
$$

The following result goes back at least as far as \cite{Agmo62}.

\begin{thm}
\label{t.order.Sobolev}
Let
   $s \in \mathbb R$ and
   $A : H^s (\cD) \to H^s (\cD)$ be a compact operator.
If there is $\delta s > 0$ such that $A$ maps $H^s (\cD)$ continuously to
                                              $H^{s + \delta s} (\cD)$,
then it belongs to Schatten class $\gS_{n/\delta s  + \varepsilon}$ for each $\varepsilon > 0$.
\end{thm}

\begin{proof}
For the case $s \in {\mathbb Z}_{\geq 0}$ see \cite{Agmo62}.
For the case $s \in \mathbb{R}$ and Sobolev spaces on a compact closed manifold $\cD$ see
Proposition 5.4.1 in \cite{Agra90}.
To the best of our knowledge, no proof has been appeared for the general case.
So we indicate crucial steps of the proof.

Let $Q$ be the cube
$$
   Q = \{ x \in \mathbb{R}^n :\,  |x_j| < \pi,\; j = 1, \ldots, n \}
$$
in $\mathbb{R}^n$.
Given a function $u \in L^2 (Q)$, we consider the Fourier series expansion
$$
   u (x)
 \sim
   \sum_{k \in \mathbb{Z}^{n}}
   c_{k} (u)\, e^{\imath \langle k, x \rangle}
$$
and introduce the norm
$$
   \| u \|_{H^{(s)}}^2
 = |a_0 (u)|^2
 + \sum_{k \in \mathbb{Z}^{n} \setminus \{ 0 \}} |k|^{2s}\, |c_k (u)|^2,
$$
where $s$ is a non-negative real number.
The subspace of functions for which this norm is finite is denoted by
   $H^{(s)}$.
Obviously, $H^{(s)}$ is a Hilbert space which, for non-negative integral $s$, can
be regarded as a closed subspace of the Sobolev space $H^{s} (Q)$.
We see readily that $H^s_{\mathrm{comp}} (Q) \mapsto H^{(s)}$.
For $s < 0$, we write $H^{(s)}$ for the dual of
                      $H^{(-s)}$
with respect to the sesquilinear pairing $(\cdot, \cdot)$ induced by the inner product
   $(\cdot, \cdot)_{H^{(0)}}$.
The norm in $H^{(s)}$ is still given by the same formula, as is easy to check.

Without loss of the generality we can assume  that the closure of $\cD$ is situated
in the cube $Q$.
For $s \geq 0$, we denote by $r_{s,\cD}$ the restriction operator from $H^s (Q)$ to
                                                                       $H^s (\cD)$.
By the above, $r_{s,\cD}$ acts to the elements of $H^{(s)}$, too, mapping these
continuously to $H^s (\cD)$.
As the boundary of $\cD$ is Lipschitz, for each $s \in {\mathbb Z}_{\geq 0}$ there
is a bounded extension operator
   $e_{s,\cD} : H^{s} (\cD) \to H^s_{\mathrm{comp}} (Q)$
(see for instance \cite[Ch.~6]{Bur98}).
We will think of $e_{s,\cD}$ as bounded linear operator from $H^s (\cD)$ to
                                                             $H^{(s)}$,
provided that $s \in {\mathbb Z}_{\geq 0}$.

Given any non-negative integer $s$, an interpolation procedure applies to the pair
   $(H^s (\cD), H^{s+1} (\cD))$,
thus giving a family of function spaces in $\cD$ of fractional smoothness
   $(1-\vartheta) s + \vartheta (s+1) = s + \vartheta$
with $0 < \vartheta < 1$.
The Banach spaces obtained in this way coincide with $H^{s + \vartheta} (\cD)$ up to
equivalent norms.
Thus, we can apply interpolation arguments to conclude that there is a bounded linear
extension operator
$
   e_{s,\cD} : H^{s} (\cD) \to H^{(s)}
$
for all real $s \geq 0$.
By construction,
\begin{equation}
\label{eq.ext}
   r_{s,\cD}\, e_{s,\cD} u = u
\end{equation}
holds for each $u \in H^s (\cD)$ with $s \geq 0$.

For $s < 0$ we introduce the mappings
$$
\begin{array}{rcrcl}
   r_{s,\cD}
 & :
 & H^{(s)}
 & \to
 & H^{s} (\cD),
\\
   e_{s,\cD}
 & :
 & H^{s} (\cD)
 & \to
 & H^{(s)},
\end{array}
$$
using the duality between the spaces $H^{(s)}$ and
                                     $H^{(-s)}$.
Namely,  if $s < 0$ we set
\begin{equation}
\label{eq.ext.r}
\begin{array}{rcl}
    (r_{s,\cD} u, v)
 & :=
 & (u, e_{-s,\cD} v),
\\
    (e_{s,\cD} u, v)
 & :=
 & (u, r_{-s,\cD} v)
\end{array}
\end{equation}
for all $u \in H^{(s)}$,
        $v \in H^{-s} (\cD)$
and for all $u \in H^{s} (\cD)$,
            $v \in H^{(-s)}$,
respectively.
As
$$
   |(u, e_{-s,\cD} v)|
 \leq
   \| u \|_{H^{(s)}}\, \| e_{-s,\cD} \|\, \| v \|_{H^{-s} (\cD)}
$$
for all $u \in H^{(s)}$ and
        $v \in H^{-s} (\cD)$,
which is a consequence of duality between the spaces $H^{(s)}$ and
                                                     $H^{(-s)}$,
the first identity of (\ref{eq.ext.r}) defines a bounded linear operator
$
   r_{s,\cD} :  H^{(s)} \to H^{s} (\cD)
$
indeed.
Similarly, by the duality between $H^{s} (\cD)$ and
                                  $H^{-s} (\cD)$
(cf. Lemma \ref{l.dual}), the second identity of (\ref{eq.ext.r}) defines
a bounded linear operator
$
   e_{s,\cD} :  H^{s} (\cD) \to H^{(s)}.
$

On applying equality (\ref{eq.ext}) we get
$
   (r_{s,\cD}\, e_{s,\cD} u, v)
 = (u, v)
$
for all $u \in H^{s} (\cD)$ and
        $v \in H^{-s} (\cD)$
with real $s < 0$.
In other words, the operators $r_{s,\cD}$ and
                              $r_{s,\cD}$
satisfy (\ref{eq.ext}) for all $s \in \R$, i.e.,
\begin{equation}
\label{eq.ext.a}
   r_{s,\cD}\, e_{s,\cD} = \Id_{H^s (\cD)}.
\end{equation}

For $t > s$ we denote by
$$
\begin{array}{rcrcl}
   \iota_{t,s,\cD}
 & :
 & H^{t} (\cD)
 & \to
 & H^{s} (\cD),
\\
   \iota_{t,s}
 & :
 & H^{(t)}
 & \to
 & H^{(s)}
\end{array}
$$
the natural inclusion mappings.
If $t < 0$, by this is meant
\begin{equation}
\label{eq.irm-}
\begin{array}{rcl}
   (\iota_{t,s,\cD} u, v)
 & =
 & (u, \iota_{-s,-t,\cD} v),
\\

   (\iota_{t,s} u, v)
 & =
 & (u, \iota_{-s,-t,\cD} v)
\end{array}
\end{equation}
for all
   $u \in H^t (\cD)$,
   $v \in H^{-s} (\cD)$
and
   $u \in H^{(t)}$,
   $v \in H^{(-s)}$,
respectively.
It is clear that
\begin{equation}
\label{eq.rie+}
\begin{array}{rcl}
   r_{s,\cD}\, \iota_{t,s}
 & =
 & \iota_{t,s,\cD}\, r_{t,\cD},
\\
   \iota_{t,s}\, e_{t,\cD}
 & =
 & e_{s,\cD}\, \iota_{t,s,\cD},
\\
   r_{s,\cD}\, \iota_{t,s}\, e_{t,\cD}
 & =
 & \iota_{t,s,\cD}
\end{array}
\end{equation}
provided $t \geq 0$.
If $t < 0$ then combining (\ref{eq.ext.r}),
                          (\ref{eq.irm-}) and
                          (\ref{eq.rie+})
yields
\begin{eqnarray*}
   (r_{s,\cD}\, \iota_{t,s} u, v)
 & =
 & (u, \iota_{-s,-t}\, e_{-s,\cD} v)
\\
 & =
 & (u, e_{-t,\cD}\, \iota_{-s,-t,\cD} v \rangle
\\
 & =
 & (\iota_{t,s,\cD}\, r_{t,\cD} u, v)
\end{eqnarray*}
for all
   $u \in H^{(t)}$ and
   $v \in H^{-s} (\cD)$,
and
\begin{eqnarray*}
   (\iota_{t,s}\, e_{t,\cD} u, v)
 & =
 & (u, r_{-t,\cD}\, \iota_{-s,-t} v)
\\
 & =
 & (u, \iota_{-s,-t}\, r_{-s,\cD} v)
\\
 & =
 & (e_{s,\cD}\, \iota_{t,s,\cD} u, v)
\end{eqnarray*}
for all
   $u \in H^{t} (\cD)$ and
   $v \in H^{(-s)}$,
whence
$
   r_{s,\cD}\, \iota_{t,s}\, e_{t,\cD} = \iota_{t,s,\cD}.
$
Therefore, equalities (\ref{eq.rie+}) are valid not only for real $t \geq 0$ but also
for all $t \in \R$.

\begin{lem}
\label{l.order.Sobolev}
Let
   $s \in \mathbb R$ and
   $K : H^{(s)} \to H^{(s)}$ be a compact operator.
If there is $\delta s > 0$ such that $K$ maps $H^{(s)}$ continuously to
                                              $H^{(s + \delta s)}$,
then $K$ is of Schatten class $\gS_{n / \delta s  + \varepsilon }$ for each $\varepsilon > 0$.
\end{lem}

\begin{proof}
Put
$$
   \mathit{\Lambda}_r u\, (x)
 = \sum_{k \in \mathbb{Z}^{n}}
   (1 + |k|^{2})^{r/2}\,
   c_{k} (u)\, e^{\imath \langle k, x \rangle}.
$$
Obvioulsy, $\mathit{\Lambda}_r$ maps $H^{(s)}$ continuously to
                                     $H^{(s-r)}$
for all $s \in \R$.
F\"{u}r each fixed $s$, the operator
   $\mathit{\Lambda}_{-r}\, \iota_{s+r,s}$
is selfadjoint and compact in $H^{(s+r)}$.
Its eigenvalues are $(1 + |k|^{2})^{- r/2}$ and the corresponding eigenfunctions are
                    $e^{\imath \langle k, x \rangle}$.
The series
$$
   \sum_{k \in \mathbb{Z}^{n}} (1+|k|^2)^{- p r / 2}
$$
(counting the eigenvalues with their multiplicities) converges for all $p > n/r$, and so
$
   \mathit{\Lambda}_{-r}\, \iota_{s+r,s}
$
is of Schatten class $\gS_{n/r + \varepsilon}$ for any $\varepsilon > 0$.

Obviously,
   $\mathit{\Lambda}_{-r} \mathit{\Lambda}_{r} = \Id$
holds for all $r > 0$.
By assumption, the operator $K$ factors through the embedding
   $\iota_{s + \delta s,s} : H^{(s + \delta s)} \to H^{(s)}$,
i.e., there is a bounded linear operator
   $K_0 : H^{(s)} \to H^{(s + \delta s)}$
such that
   $K = \iota_{s + \delta s,s} K_0$.
Then
\begin{eqnarray*}
   K
 & = &
   \mathit{\Lambda}_{- \delta s}\, \mathit{\Lambda}_{\delta s} K
\\
 & = &
   \mathit{\Lambda}_{-\delta s}\, \mathit{\Lambda}_{\delta s}\, \iota_{s+\delta s,s}
   K_0
\\
 & = &
   \mathit{\Lambda}_{-\delta s}\, \iota_{s + \delta s,s}\, \mathit{\Lambda}_{\delta s}
   K_0.
\end{eqnarray*}
Since the operator
   $\mathit{\Lambda}_{\delta s}\, K_0 : H^{(s)} \to H^{(s)}$
is bounded, Lemma \ref{l.2} implies that $K$ belongs to the Schatten class
   $\gS_{n/\delta s + \varepsilon}$
for any $\varepsilon > 0$.
\end{proof}

We are now in a position to complete the proof of Theorem \ref{t.order.Sobolev}.
Suppose that
   $A_0 : H^{s} (\cD) \to  H^{s + \delta s} (\cD)$
is a bounded linear operator, such that $A = \iota_{s + \delta s,s} A_0$.
Set
$$
   K_0 = e_{s + \delta s,\cD}\, A_0\, r_{s,\cD},
$$
then $K_0$ maps $H^{(s)}$ continuously to $H^{(s + \delta s)}$.
By Lemma \ref{l.order.Sobolev}, the composition
   $K = \iota_{s + \delta s,s} K_0$
is of Schatten class $\gS_{n/\delta s + \varepsilon}$ for any $\varepsilon > 0$.
Besides, we get
\begin{equation}
\label{eq.AB}
   r_{s + \delta s,\cD}\, K_0 =  A_0\, r_{s,\cD}
\end{equation}
because of (\ref{eq.ext.a}).

Let
   $\lambda$ be a non-zero eigenvalue of $A$
and
   $u \in H^s (\cD)$ a root function corresponding to $\lambda$,
i.e.,
   $(A - \lambda \Id)^m u = 0$
for some natural number $m$.
Then, using (\ref{eq.rie+}) and
            (\ref{eq.AB}),
we conclude that
\begin{eqnarray*}
   (K - \lambda \Id)^m\, e_{s,\cD} u
 & = &
   e_{s,\cD}\, (A - \lambda \Id)^m  u
\\
 & = &
   0,
\end{eqnarray*}
that is each non-zero eigenvalue of $A$ is actually an eigenvalue for $K$ of the same
multiplicity.
Therefore, $A$ belongs to the Schatten class
   $\gS_{n/\delta s + \varepsilon}$
for any $\varepsilon > 0$, too.
\end{proof}

\begin{cor}
\label{c.order.SL}
If for some $s > 0$ there is a continuous embedding
\begin{equation}
\label{eq.emb.s}
   \iota_s : H_{\SL} (\cD) \hookrightarrow H^s (\cD)
\end{equation}
then any compact operator
   $R : H^-_{\SL} (\cD)  \to H^-_{\SL} (\cD)$
which maps $H^-_{\SL} (\cD)$ continuously to
           $H_{\SL} (\cD)$
is of Schatten class $\gS_{n/2s + \varepsilon}$ for any $\varepsilon > 0$.
In particular, its order is finite.
\end{cor}

\begin{proof}
We first observe that the operator $\iota_s$ induces via composition a bounded 
inclusion operator
   $\iota_s' : H^{-s} (\cD) \to H^-_{\SL} (\cD)$.
This latter is actually the transpose of $\iota_s$.   

Let $R_0$ be the operator $R$ which is thought of as a bounded map of $H^-_{\SL} (\cD)$ to
                                                                      $H_{\SL} (\cD)$.
Then the operator $\iota_s R_0 \iota_s'$ maps $H^{-s} (\cD)$ continuously to
                                              $H^s (\cD)$.

Write
   $i : H^s (\cD) \hookrightarrow L^2 (\cD)$
for the natural inclusion and
   $i' : L^2 (\cD) \hookrightarrow H^{-s} (\cD)$
for the corresponding map induced by duality.
It follows from Theorem \ref{t.order.Sobolev} that the operator
   $i' i\, \iota_s R_0 \iota_s' : H^{-s} (\cD) \to H^{-s} (\cD)$
is of Schatten class
   $\gS_{n/2s + \varepsilon}$
for any $\varepsilon  > 0$.
By assumption, the embedding
   $\iota : H_{\SL} (\cD) \hookrightarrow L^2 (\cD)$
factors through $i$, i.e., $\iota = i  \circ \iota_s$ whence
$
   \iota' = \iota_s'  i'.
$
Hence
$
   R = \iota' \iota\, R_0
     = \iota_s'\, i' \, i \, \iota_s\, R_0.
$

Let now
   $\lambda$ be a non-zero eigenvalue of $R$
and
   $u$ be a root function of $R$ corresponding to $\lambda$,
i.e.,
   $(R - \lambda \Id)^m u = 0$
for a natural number $m$.
Then it follows from the binomial formula that $u$ belongs to the image of
the operator $\iota_s'$, i.e.,
   $u = \iota_s' u_0$
for some $u_0 \in H^{-s} (\cD)$.
Hence
\begin{eqnarray*}
   (R - \lambda \Id)^m u
 & = &
   (R - \lambda \Id)^m \iota_s' u_0
\\
 & = &
   \iota_s'\, (i'\, i\, \iota_s\, R_0 \iota_s' - \lambda \Id)^m  u_0
\\
 & = &
   0.
\end{eqnarray*}

As the operator $\iota'$ is injective, each  eigenvalue of the operator $R$ is in
fact an eigenvalue of  $i'\, i\, \iota_s\, R_0 \iota_s'$ of the same multiplicity.
Therefore $R$ lies in $\gS_{n/2s + \varepsilon}$ for any $\varepsilon > 0$, too.
\end{proof}

\begin{cor}
\label{c.order.SL0}
Suppose there is a continuous embedding (\ref{eq.emb.s}) with some $s > 0$.
Then the operators $Q_1$,
                   $Q_2$ and
                   $Q_3$
are of Schatten class $\gS_{n/2s + \varepsilon}$ for any $\varepsilon  > 0$
   (and so they are of finite order).
\end{cor}

\begin{proof}
Since $Q_1 = \iota' \iota\, L^{-1}_{0}$ and
      $L^{-1}_{0}$ maps $H_{SL} (\cD)$ continuously to $H^-_{SL} (\cD)$,
the operator $Q_1$ is of Schatten class
   $\gS_{n/2s + \varepsilon}$
for any $\varepsilon  > 0$, which is due to Corollary \ref{c.order.SL}.
On the other hand, Lemma \ref{l.L0.selfadj} shows that the operators $Q_1$,
                                                                     $Q_2$ and
                                                                     $Q_3$
have the same eigenvalues.
Hence, all these operators belong to the Schatten class
   $\gS_{n/2s + \varepsilon}$
for any $\varepsilon  > 0$, as desired.
\end{proof}

Now we want to study the completeness of root functions of ``small'' perturbations of compact self-adjoint operators instead of the weak ones.
To this end we apply the so-called method of rays of minimal growth of resolvent which leads to more general results than Theorem \ref{t.keldysh}.
This idea seems to go back at least as far as \cite{Agmo62}.

\section{Rays of minimal growth}
\label{s.3}

We first describe briefly the method of minimal growth rays following
   \cite{DunfSchw63} and
   Theorem 6.1 of \cite[p.~302]{GokhKrei69}.

Let $L : H_{\SL} (\cD) \to H^-_{\SL} (\cD)$ be the bounded linear operator constructed in Section~\ref{s.4}.
We still assume that estimates (\ref{eq.incl}) and
                               (\ref{eq.positive.part1})
hold and that the operator $L$ is Fredholm.
In the sequel we confine ourselves to those Sturm-Liouville problems for which the spectrum of the corresponding unbounded closed operator $T : H^-_{\SL} (\cD) \to H^-_{\SL} (\cD) $ is discrete,
   cf. \cite{Agmo62}.
We denote by $\mathcal{R} (\lambda; T) $ the resolvent of the operator $T$.

\begin{defn}
\label{d.romg}
A ray $\arg \lambda = \vartheta$ in the complex plane $\C$ is called a ray of minimal growth of the resolvent
$
   \mathcal{R} (\lambda; T) : H^-_{\SL} (\cD) \to H^-_{\SL} (\cD)
$
if
   the resolvent exists for all $\lambda$ of sufficiently large modulus on
   this ray,
and if, moreover,
   for all such $\lambda$ an estimate
\begin{equation}
\label{eq.13}
   \| \mathcal{R} (\lambda; T) \|_{\mathcal{L} (H^-_{\SL} (\cD))}
 \leq
   c\, |\lambda|^{-1} %
\end{equation}
holds with a constant $C > 0$.
\end{defn}

\begin{thm}
\label{p.root.Schatten}
Let the space $H_{\SL} (\cD)$ be continuously embedded into $H^s (\cD)$ for some $s > 0$.
Suppose there are rays of minimal growth of the resolvent
   $\arg \lambda = \vartheta_j$,
where $j = 1, \ldots, J$, in the complex plane, such that the angles between any two
neigh\-bouring rays are less than
   $2 \pi s / n$.
Then the spectrum of the operator $T$ is discrete and the root functions form a complete
system in the spaces $H^-_{\SL} (\cD)$,
           $L^2 (\cD)$ and
           $H_{\SL} (\cD)$.
\end{thm}

\begin{proof}
The proof actually follows by the same method as that in Theorem 3.2 of \cite{Agmo62}.

Since the spectrum of the operator $T$ is different from the whole complex plane it is
actually discrete.
It remains to show that if $g \in H^-_{\SL} (\cD)$ is orthogonal to all eigen- and associated
functions of the operator $T$ then $g$ is identically zero.
By the Hahn-Banach theorem, this implies that the root functions of $T$ are complete in
$H^-_{\SL} (\cD)$.

Since the operators $T$ and
                    $T - \lambda_0 \Id$
have the same root functions, we may assume without loss of generality that the origin is
not in the spectrum of $T$.
Choosing $\lambda_0 = 0$ in $R = (T - \lambda_0 \Id)^{-1}$, we set $R = T^{-1}$.

Consider now the function
\begin{equation}
\label{eq.14}
   f (\lambda)
 = \Big( \mathcal{R} \Big( \frac{1}{\lambda}; R \Big) u, g \Big)_{H_{\SL}^- (\cD)},
\end{equation}
where
   $u \in H^-_{\SL} (\cD)$
and
   $(\cdot,\cdot)_{H_{\SL}^- (\cD)}$ stands for the scalar product in $H^-_{\SL} (\cD)$.
Since the resolvent of $R$ is a meromorphic function with poles at the points of the spectrum
of $R$, the function $f$ is analytic for $\lambda \neq \lambda_\nu$, where
   $\{ \lambda_\nu \}$
is the sequence of eigenvalues of $R^{-1} = T$.
We shall use a familiar relation between the resolvents of the operators $T$ and
                                                                         $T^{-1}$, namely
\begin{equation}
\label{eq.15}
   \mathcal{R} \Big( \frac{1}{\lambda}; T^{-1} \Big)
 = - \lambda I - \lambda^2 \mathcal{R} (\lambda; T).
\end{equation}

Consider the expansion
$$
   \mathcal{R} (\lambda; T) u
 = \frac{f_{-N}}{(\lambda-\lambda_\nu)^N}
 + \frac{f_{-N+1}}{(\lambda-\lambda_\nu)^{N-1}}
 + \ldots
 + \frac{f_{-1}}{\lambda-\lambda_\nu}
 + \sum_{k = 0}^\infty f_k\, (\lambda-\lambda_\nu)^k,
$$
in a neighbourhood of the point $\lambda = \lambda_\nu$,
   where $\lambda_\nu$ is a pole of $\mathcal{R} (\lambda; T)$.
Here
   $N \geq 1$ and $f_{-N} \neq 0$,
   the functions $f_{-N}, \ldots, f_{-1} \in H^-_{\SL} (\cD)$ form a chain of associated
   functions of $T$,
and
   $f_k \in H^- _{\SL} (\mathcal{D})$ for $k \geq 0$.
This expansion implies that $\lambda_\nu$ is a regular point of $f (\lambda)$, for $g$ is
orthogonal to all $f_{-N}, \ldots, f_{-1}$.
Therefore, $f (\lambda)$ is an entire function.

Relations (\ref{eq.13}),
          (\ref{eq.14}) and
          (\ref{eq.15})
imply that $f$ is of exponential type, i.e., there is a constant $c > 0$, such that
\begin{equation}
\label{eq.16}
   |f (\lambda)|
 \leq
   c\, \exp |\lambda|
\end{equation}
for $|\lambda| \to \infty$,
   provided that $\arg \lambda = \vartheta_j$ for some $j = 1, \ldots, J$.
We use the following lemma taken from \cite{DunfSchw63}.

\begin{lem}
\label{l.1}
Assume that $R$ is a compact linear operator of Schatten class $\gS_p$,
   with $0 < p < \infty$,
in a Hilbert space $H$.
Then there exists a sequence $\rho_j$ satisfying $\rho_j \to 0$, such that
$$
   \| \mathcal{R} (\lambda; R) \|_{\mathcal {L} (H)}
 \leq
   \mathrm{const}\, \exp (c\, |\lambda|^{-p})
$$
for $|\lambda| = \rho_j$.
\end{lem}

According to Corollary \ref{c.order.SL}, the operator $R$ belongs to
   $\gS_{n/2s + \varepsilon}$
for any $\varepsilon > 0$.
Then it follows from Lemma \ref{l.1} that for any $\varepsilon > 0$ there exists a
sequence
   $\rho_j \to 0$,
such that
\begin{equation}
\label{eq.17}
   |f (\lambda)|
 \leq
   \exp \Big( |\lambda|^{-\frac{\scriptstyle n}
                               {\scriptstyle 2 s} - \varepsilon} \Big)
\end{equation}
for all $\lambda \in \C$ satisfying $|\lambda| = 1/\rho_j$.

Consider $f (\lambda)$ in the closed corner between the rays
   $\arg \lambda = \vartheta_j$ and
   $\arg \lambda = \vartheta_{j+1}$.
Its angle is less than $2 \pi s / n$.
Since
$$
   \mathcal{R} \Big( \frac{1}{\lambda}; R \Big)
 = - \lambda \Id - \lambda^2 \mathcal{R} (\lambda; T)
$$
and each ray $\arg \lambda = \vartheta_j$ is a ray of minimal growth,
   inequality (\ref{eq.16}) is fulfilled on the sides of the corner
and
   inequality (\ref{eq.17}) on a sequence of arcs which tends to infinity.

Choosing $\varepsilon > 0$ in (\ref{eq.17}) sufficiently small and applying the
Fragmen-Lindel\"{o}f theorem we conclude that
   $|f (\lambda)| = O (|\lambda|)$
as $|\lambda| \to \infty$ in the whole complex plane.
Therefore, $f (\lambda)$ is an affine function, i.e.,
   $f (\lambda) = a_{0} + c_{1} \lambda$.
On the other hand, we have
$$
   \mathcal{R} \Big( \frac{1}{\lambda}; R \Big)
 = - \lambda\, \Id - \lambda^2\, R + \ldots,
$$
and so
$$
   f (\lambda)
 = - \lambda\, (u,g)_{H_{\SL}^- (\cD)} - \lambda^2\, (Ru, g)_{H_{\SL}^- (\cD)} + \ldots.
$$
Since $f (\lambda)$ is affine, we get
$$
   (Ru, g)_{H_{\SL}^- (\cD)} = 0
$$
for all $u \in H^-_{\SL} (\cD)$.
Hence it follows that $g = 0$, for the range of the operator $R$ is dense in
$H^-_{\SL} (\cD)$.
Thus, the system of root functions of the operator $T$ is complete in $H^-_{\SL} (\cD)$.

As the operators $T$,
                 $(T - \lambda_0 \Id)$ and
                 $(T - \lambda_0 \Id)^{-1}$
have the same root functions, it suffices to repeat the arguments of the proof of
Theorem \ref{p.root.func.1} to see the completeness in the spaces $L^2 (\cD)$ and
                                                                  $H_{\SL} (\cD)$.
\end{proof}

This  theorem raises the question under what conditions neighbouring rays of minimal growth
are close enough.
We now indicate some conditions for a ray $\arg \lambda = \vartheta$ in the complex plane to
be a ray of minimal growth for the resolvent of $T$.

\begin{lem}
\label{l.rays.T0}
Each ray $\arg \lambda = \vartheta$ with $\vartheta \neq 0$ is a ray of minimal growth for
   $\mathcal{R} (\lambda; T_0)$
and
\begin{equation}
\label{eq.rays.T01}
   \| (T_0 - \lambda \Id)^{-1} \|_{\mathcal{L} (H^-_{\SL} (\cD))}
 \leq
   \left\{ \begin{array}{lcl}
             (|\lambda|\, |\sin (\arg \lambda)|)^{-1},
           & \mbox{if}
           & |\arg \lambda| \in (0, \pi/2),
\\
             |\lambda|^{-1},
           & \mbox{if}
           & |\arg \lambda| \in [\pi/2, \pi].
           \end{array}
   \right.
\end{equation}
Moreover, the operator
   $L_0 - \lambda\, \iota' \iota : H_{\SL} (\cD) \to H^-_{\SL} (\cD)$
is continuously invertible and
\begin{equation}
\label{eq.rays.T02}
   \| (L_0 - \lambda\, \iota' \iota )^{-1} \|_{\mathcal{L} (H^-_{\SL} (\cD), H_{\SL} (\cD))}
 \leq
   \left\{ \begin{array}{lcl}
             |\sin (\arg \lambda)|,
           & \mbox{if}
           & |\arg \lambda| \in (0, \pi/2),
\\
             1,
           & \mbox{if}
           & |\arg \lambda| \in [\pi/2, \pi].
           \end{array}
   \right.
\end{equation}
\end{lem}

\begin{proof}
According to Lemma \ref{l.spectrum.T0} the resolvent
$$
   (T_0 - \lambda \Id)^{-1} : H_{\SL}^- (\cD) \to H_{\SL}^- (\cD)
$$
exists for all $\lambda \in \C$ away from the positive real axis.
As the operator $Q_3 = L^{-1}_0\, \iota' \iota$ is self-adjoint, the operator $T_0$ is
symmetric, i.e.,
\begin{eqnarray*}
   (T_0 u, g)_{H_{\SL}^- (\cD)}
 & = &
   (L_0 u, g)_{H_{\SL}^- (\cD)}
\\
 & = &
   (u, Q_{3}^{-1} g)_{\SL}
\\
 & = &
   (Q_{3}^{-1} u, g)_{\SL}
\\
 & = &
   (u, L_0 g)_{H_{\SL}^- (\cD)}
\\
 & = &
   (u, T_0 g)_{H_{\SL}^- (\cD)}
\end{eqnarray*}
for all $u, v \in H_{\SL} (\cD)$.
If $|\arg{(\lambda)}| \in (0, \pi/2)$, then
\begin{eqnarray*}
   \| (T_0 - \lambda \Id) u \|_{H_{\SL}^- (\cD)}^2
 & = &
   \| (T_0 - \Re \lambda\, \Id) u \|^2_{H_{\SL}^- (\cD)}
 + |\Im \lambda|^2\, \| u \|^2_{H_{\SL}^- (\cD)}
\\
 & \geq &
   |\lambda|^2 |\sin (\arg \lambda)|^2\, \| u \|_{H_{\SL}^- (\cD)}^2
\end{eqnarray*}
for all $u \in  H_{\SL} (\cD)$,
   which establishes the first estimate of (\ref{eq.rays.T01}).
If $|\arg \lambda| \in [\pi/2, \pi]$, then $\Re \lambda \leq 0$ whence
$$
   \| (T_0 - \lambda \Id) u \|^2_{H_{\SL}^- (\cD)}
 \geq |\lambda|^2\, \| u \|^2_{H_{\SL}^- (\cD)}
$$
and so the second estimate of (\ref{eq.rays.T01}) holds.

Now it follows from (\ref{eq.T-mu}) that the operator $L_0 - \lambda\,  \iota' \iota$
is injective for $\lambda \in \C$ away from the positive real axis.
As this operator is Fredholm and its index is zero, it is continuously invertible.
Finally, as the operator $Q_3 = L_0^{-1}\, \iota' \iota$ is positive, we deduce readily that
\begin{eqnarray*}
   \|(L_0 - \lambda\, \iota' \iota) u \|_{H_{\SL}^- (\cD)}
 & = &
   \| (\Id - \lambda L_0^{-1}\, \iota' \iota) u \|_{\SL}
\\
 & \geq &
   |\lambda|\, |\Im \lambda^{-1}|\,  \| u \|_{\SL}
\\
 & = &
   |\sin (\arg \lambda)|\, \| u \|_{\SL},
\end{eqnarray*}
if $|\arg \lambda| \in (0, \pi/2)$, i.e., the second estimate of (\ref{eq.rays.T02}) is
fulfilled.
Similar arguments lead to the second estimate of (\ref{eq.rays.T02}).
\end{proof}

\begin{thm}
\label{t.root.func.Schatten}
Let
   the space $H_{\SL} (\cD)$ be continuously embedded into $H^s (\cD)$ for some $s > 0$
and
   estimate (\ref{eq.positive.part1}) be fulfilled with a constant $c < |\sin (\pi s / n)|$.
Then all eigenvalues of the closed operator
$
   T :  H_{\SL}^- (\cD) \to H_{\SL}^- (\cD)
$
belong to the corner
$
   |\arg \lambda| \leq \arcsin c,
$
each ray $\arg \lambda = \vartheta$ with $|\vartheta| > \arcsin c$ is a ray of minimal
growth for $\mathcal{R} (\lambda; T)$ and the system of root functions is complete in the
spaces
   $H^-_{\SL} (\cD)$,
   $L^2 (\cD)$ and
   $H_{\SL} (\cD)$.
\end{thm}

\begin{proof}
First we note that, by Lemma \ref{l.solv.SL2}, the operator
   $L : H_{\SL} (\cD) \to H^-_{\SL} (\cD)$
is invertible.
Indeed, $L = L_0 + \delta L$ where  $\delta L : H_{\SL} (\cD) \to H^-_{\SL} (\cD)$ is a
bounded operator with the norm
$
   \| \delta L \|_{\mathcal{L} (H_{\SL} (\cD), H^-_{\SL} (\cD))}
 < 1
 = \| L^{-1}_0 \|^{-1}.
$
In particular, by (\ref{eq.T-mu}), the spectrum of the corresponding operator $T$ does not
coincide with the whole complex plane.

Fix $\vartheta \neq 0$ and set
   $m_\vartheta = |\sin \vartheta|$, if $|\vartheta| \in (0,\pi/2)$, and
   $m_\vartheta = 1$, if $|\vartheta| \in [\pi/2,\pi]$.
If $m_\vartheta > c$ then
$$
   \| \delta L \|_{\mathcal{L} (H_{\SL} (\cD), H^-_{\SL} (\cD))}
 \leq
   c
 <
   m_\vartheta
 \leq
   \| (L_0 - \lambda\, \iota' \iota)^{-1} \|_{\mathcal{L} (H_{\SL} (\cD), H^-_{\SL} (\cD))}^{-1}.
$$
Hence it follows that the operator
$
   L - \lambda\, \iota' \iota : H_{\SL} (\cD) \to H^-_{\SL} (\cD)
$
is continuously invertible and
\begin{equation}
\label{eq.bound.c1}
   \|(L - \lambda\, \iota' \iota)^{-1} \|_{\mathcal{L} (H^-_{\SL} (\cD), H_{\SL} (\cD))}
 \leq
   (m_\vartheta - c)^{-1}.
\end{equation}

In order to establish estimate (\ref{eq.13}) we have to show that there is a constant $C > 0$,
such that
$$
   C\, |\lambda|^{-1}\, \| (T - \lambda \Id) u \|_{H_{\SL}^- (\cD)}
 \geq
   \| u \|_{H_{\SL}^- (\cD)}
$$
for all $u \in H_{\SL} (\cD)$.

If
   $\arg \lambda = \vartheta$
with $m_\vartheta > c$, then, by (\ref{eq.T-mu}), we get
\begin{eqnarray*}
   \| (T - \lambda \Id) u \|_{H_{\SL}^- (\cD)}
 & = &
   \| (L - \lambda\, \iota' \iota) u \|_{H_{\SL}^- (\cD)}
\\
 & \geq &
   (m_\vartheta - c)\, \| u \|_{\SL}
\\
 & \geq &
   (m_\vartheta) - c)\, \| u \|_{H_{\SL}^- (\cD)}
\end{eqnarray*}
for all $u \in H_{\SL} (\cD)$.
Therefore, given any $\lambda$ on the ray $\arg \lambda = \vartheta$ with $m_\vartheta > c$,
it follows that

1)
The range of the operator $T - \lambda \Id : H^-_{\SL} (\cD) \to H^-_{\SL} (\cD)$ is a
closed subspace of $H^-_{\SL} (\cD)$.

2)
The null space of the operator $T - \lambda \Id : H^-_{\SL} (\cD) \to H^-_{\SL} (\cD)$ is
trivial.

By (\ref{eq.T-mu}) the range of $T - \lambda \Id$ coincides with the range of
   $L - \lambda\, \iota' \iota$
which is the whole space $H^-_{\SL} (\cD)$.
Hence, the resolvent $(T - \lambda \Id)^{-1}$ exists for all $\lambda$ away from the corner
   $|\arg \lambda| \leq \arcsin c$
in the complex plane.
On applying (\ref{eq.T-mu}) and
            Lemma \ref{l.rays.T0}
we obtain
\begin{equation}
\label{eq.T}
   T - \lambda \Id
 = L_0 + \delta L - \lambda\, \iota' \iota
 = (\Id + \delta L (L_0 - \lambda\, \iota' \iota)^{-1}) (T_0 - \lambda \Id)
\end{equation}
on $H_{\SL} (\cD)$ and
\begin{eqnarray*}
\lefteqn{
   \| (\Id + \delta L (L_0 - \lambda\, \iota' \iota)^{-1}) u \|_{H_{\SL}^- (\cD)}
}
\\
 & \geq &
   \| u \|_{H_{\SL}^- (\cD)}
 - \| \delta L \|_{\mathcal{L} (H_{\SL} (\cD), (H^-_{\SL} (\cD)))}\,
   \| (L_0 - \lambda\, \iota' \iota)^{-1} u \|_{H_{\SL} (\cD)}
\\
 & \geq &
   (1 - c / m_{\vartheta})\, \| u \|_{H_{\SL}^- (\cD)}.
\end{eqnarray*}
Therefore the operator
   $\Id + \delta L (L_0 - \lambda\,  \iota' \iota)^{-1}$
is continuously invertible as Fredholm operator of zero index and trivial null space.
Moreover,
$$
   \| (\Id + \delta L (L_0 - \lambda\, \iota' \iota)^{-1})^{-1} \|_{\mathcal{L} (H^-_{\SL} (\cD))}
 \leq
   (1 - c / m_\vartheta)^{-1}.
$$
Now (\ref{eq.T}) implies
\begin{eqnarray}
\label{eq.bound.T1}
\lefteqn{
   \|(T - \lambda \Id)^{-1} \|_{\mathcal{L} (H^-_{\SL} (\cD))}
}
\nonumber
\\
 & \leq &
   \| (\Id + \delta L (L_0 - \lambda\, \iota' \iota)^{-1})^{-1} \|_{\mathcal{L} (H^-_{\SL} (\cD))}
   \| (T_0 - \lambda\, \Id)^{-1} \|_{\mathcal{L} (H^-_{\SL} (\cD))}
\\
 & \leq &
   (1 - c / m_\vartheta)^{-1} m_\vartheta^{-1} |\lambda|^{-1}
\end{eqnarray}
for all $\lambda$ satisfying $\arg \lambda = \vartheta$ with $m_\vartheta > c$.

Thus, all rays outside of the corner
$
   |\arg \lambda| \leq \arcsin c
$
are rays of minimal growth.
By the hypothesis of the theorem, the angles between the pairs of neighbouring rays
   $\arg \lambda = \vartheta$,
are less than $2 \pi s / n$, and so the completeness of root functions follows from
   Theorem \ref{p.root.Schatten}.
\end{proof}

We are now in a position to prove the main result of this section.
When compared with \cite{Agra11c} our contribution consists in developing dual function
spaces which fit the problem.

\begin{thm}
\label{t.root.func.Schatten.Comp}
Let
   the space $H_{\SL} (\cD)$ be continuously embedded into $H^s (\cD)$ for some $s > 0$,
$
   \delta L : H_{\SL} (\cD) \to H^-_{\SL}(\cD)
$
be a bounded linear operator whose norm is less then $|\sin (\pi s / n)$,
and
$
   C : H_{\SL} (\cD) \to H^-_{\SL} (\cD)
$
be compact.
Then the following is true:

1)
The spectrum of the operator $T$ in $H^-_{\SL} (\cD)$ corresponding to $L_0 + \delta L + C$
is discrete.

2)
For any $\varepsilon > 0$, all eigenvalues of the operator $T$ (except for a finite number)
belong to the corner
$
   |\arg \lambda)| < \arcsin \| \delta L \| + \varepsilon.
$

3)
Each ray $\arg \lambda = \vartheta$ with
\begin{equation}
\label{eq.rays.C1}
   |\vartheta| > \arcsin \| \delta L \| %_{\mathcal{L} (H_{\SL} (\cD), H^-_{\SL} (\cD))}
\end{equation}
is a ray of minimal growth for $\mathcal{R} (\lambda; T)$.

4)
The system of root functions is complete in the spaces $H^-_{\SL} (\cD)$,
                                                       $L^2 (\cD)$ and
                                                       $H_{\SL} (\cD)$.
\end{thm}

\begin{proof}
First we note that the operator $L_0 + \delta L : H_{\SL} (\cD) \to H^-_{\SL} (\cD)$ is
continuously invertible and hence the operator
   $L_0 + \delta L + C : H_{\SL} (\cD) \to H^-_{\SL} (\cD)$
is actually Fredholm.

Theorem \ref{t.root.func.Schatten} implies that all rays satisfying (\ref{eq.rays.C1}) are
rays of minimal growth for $\mathcal{R} (\lambda; T_0 + \delta T)$ with the closed operator
   $T_0 + \delta T$
in $H^-_{\SL} (\cD)$ corresponding to
   $L_0 + \delta L : H_{\SL} (\cD) \to H^-_{\SL} (\cD)$.

Fix an arbitrary $\varepsilon > 0$.
Then estimates (\ref{eq.bound.c1}) and
               (\ref{eq.bound.T1})
imply that there are constants $c_1$ and $c_2$ depending on $\varepsilon$, such that
\begin{eqnarray}
\label{eq.bound.c11}
   \| (L_0 + \delta L - \lambda\, \iota' \iota)^{-1} \|_{\mathcal{L} (H^-_{\SL} (\cD), H_{\SL} (\cD))}
 & \leq &
   c_1,
\\
   \| (T_0 + \delta T - \lambda \Id)^{-1} \|_{\mathcal{L} (H^-_{\SL} (\cD))}
 & \leq &
   c_2\, |\lambda|^{-1}
\end{eqnarray}
for all $\lambda$ satisfying
\begin{equation}
\label{eq.rays.C2}
   |\arg \lambda|
 \geq
   \arcsin \| \delta L \|_{\mathcal{L} (H_{\SL} (\cD), H^-_{\SL} (\cD))} + \varepsilon.
\end{equation}
Then, using (\ref{eq.T-mu}),
            (\ref{eq.bound.c11}) and
            Theorem \ref{t.root.func.Schatten}
we obtain
\begin{equation}
\label{eq.rays.T8}
   T - \lambda \Id
 = \left( \Id + C (L_0 + \delta L - \lambda\, \iota' \iota)^{-1} \right)
   \left( T_0 + \delta T - \lambda \Id \right)
\end{equation}
on $H_{\SL} (\cD)$ for all rays satisfying (\ref{eq.rays.C2}).

We now prove that there is a constant $M_\varepsilon > 0$ depending on $\varepsilon$, such
that the operator
   $\Id + C (L_0 + \delta L - \lambda\, \iota' \iota)^{-1}$
is injective for all $\lambda$ satisfying both (\ref{eq.rays.C2}) and
                                               $|\lambda| \geq M_\varepsilon$.
To do this, we argue by contradiction in the same way as in the proof of
   Theorem \ref{t.root.func.1}.
Suppose for each natural number $k$ there are
   $f_k \in H^-_{\SL} (\cD)$, satisfying $\| f_k \|_{H_{\SL}^- (\cD)} = 1$, and
   $\lambda_k$, satisfying (\ref{eq.rays.C2}) and $|\lambda_k| \geq k$,
such that
\begin{equation}
\label{eq.ceee}
   \left( \Id + C (L_0 + \delta L - \lambda_k\, \iota' \iota)^{-1} \right) f_k = 0.
\end{equation}
It follows from (\ref{eq.bound.c11}) that the sequence
   $u_k = (L_0 + \delta L - \lambda_k \iota' \iota)^{-1} f_k$
is bounded in $H_{\SL} (\cD)$.
By the weak compactness principle for Hilbert spaces one can assume without restriction of
generality that the sequences  $\{ f_k \}$ and
                               $\{ u_k \}$
converge weakly in the spaces $H^-_{\SL} (\cD)$ and
                              $H_{\SL} (\cD)$
to functions $f$ and $u$, respectively.
Since $C$ is compact, it follows that the sequence $\{ C u_k \}$ converges to $Cu$ in
$H^-_{\SL} (\cD)$ and so $\{ f_k \}$ converges to $f$, which is due to (\ref{eq.ceee}).
Obviously, the $H_{\SL}^- (\cD)\,$-norm of $f$ just amounts to $1$.
In particular, we conclude that
$
   \{ C (L_0 + \delta L - \lambda_k\, \iota' \iota)^{-1}) f_k \}
$
converges to $- f$ whence
\begin{equation}
\label{eq.w000}
   f = - Cu.
\end{equation}

Further,
   as $f_k = (L_0 + \delta L - \lambda_k\, \iota' \iota)\, u_k$,
letting $k \to \infty$ in this formula yields readily
$$
   f = (L_0 + \delta L) u - \lim_{k \to \infty} \lambda_k\, \iota' \iota\,  u_k.
$$
As the operator $\iota' \iota$ is compact, the sequence
   $\{ \iota' \iota \ u_k \}$
converges to $\iota' \iota  \ u$ in the space $H^-_{\SL} (\cD)$, and
   $\iota' \iota u \neq 0$ because of (\ref{eq.w000}) and the injectivity of $\iota' \iota$.
Therefore, the weak limit
$$
   \lim_{k \to \infty} \lambda_k\, \iota' \iota\, u_k = (L_0 + \delta L) u - f
$$
fails to exist, for $\{ \lambda_k \}$ is unbounded.
A contradiction.

As the operator
   $\Id + C (L_0 + \delta L - \lambda\, \iota' \iota)^{-1}$
is Fredholm and it has index zero, this operator is continuously invertible for all
$\lambda \in \C$ satisfying both (\ref{eq.rays.C2}) and
                                 $|\lambda| \geq M_\varepsilon$.
Set
$$
   N_\varepsilon
 = \inf\,
   \| (\Id + C (L_0 + \delta L - \lambda\, \iota' \iota)^{-1}) f \|_{H_{\SL}^- (\cD)}
 \geq
   0,
$$
the infimum being over all
   $f \in H_{\SL}^- (\cD)$ of norm $1$
and all $\lambda \in \C$ satisfying (\ref{eq.rays.C2}) and
                                    $|\lambda| \geq M_\varepsilon$.
We claim that $N_\varepsilon > 0$.
To show this, we argue by contradiction.
If $N_\varepsilon = 0$ then there are sequences
   $\{ f_k \}$ in $H^-_{\SL} (\cD)$, each $f_k$ being of norm $1$,
and
   $\{ \lambda_k \}$ satisfying (\ref{eq.rays.C2}) and
                                $|\lambda| \geq M_\varepsilon$,
such that
\begin{equation}
\label{eq.ce1}
   \lim_{k \to \infty}
   \| (\Id + C (L_0 + \delta L - \lambda_k\, \iota' \iota)^{-1}) f_k \|_{H_{\SL}^- (\cD)}
 = 0.
\end{equation}
Again, by (\ref{eq.bound.c11}), the sequence
   $u_k = (L_0 + \delta L - \lambda_k\, \iota' \iota)^{-1} f_k$
is bounded in $H_{\SL} (\cD)$.
By the weak compactness principle for Hilbert spaces we may assume that the sequences
   $\{ f_k \}$ and
   $\{ u_k \}$
are weakly convergent in the spaces $H^-_{\SL} (\cD)$ and
                                    $H_{\SL} (\cD)$
to functions $f$ and
             $u$,
respectively.
Since $C$ is compact, the sequence $\{ C u_k \}$ converges to $Cu$ in $H^-_{\SL} (\cD)$ and so
                                   $\{ f_k \}$ converges to $f$
because of (\ref{eq.ce1}); obviously, $\| f \|_{H_{\SL}^- (\cD)} = 1$.
In particular, we deduce that the sequence
   $C (L_0 + \delta L - \lambda_k, \iota' \iota)^{-1})\, f_k$
converges to $- f$ whence
\begin{equation}
\label{eq.w01}
   f = - Cu
\end{equation}
with $u \neq 0$.

If the sequence $\{ \lambda_k \}$ is bounded in $\C$, then using the weak compactness principle
and passing to a subsequence, if necessary, we may assume that
    $\{ \lambda_k \}$ converges to $\lambda_0 \in \C$
which satisfies (\ref{eq.rays.C2}) and
                $|\lambda| \geq M_\varepsilon$.
Since
\begin{eqnarray*}
\lefteqn{
   (L_0 + \delta L - \lambda_k\, \iota' \iota)^{-1} f_k
 - (L_0 + \delta L - \lambda_0\, \iota' \iota)^{-1} f
}
\\
 \! & \! = \! & \!
   (L_0 \! + \! \delta L \! - \! \lambda_j\, \iota' \iota)^{-1}) (f_k \! - \! f)
 + \left( (L_0 \! + \! \delta L \! - \! \lambda_k\, \iota' \iota)^{-1}
        - (L_0 \! + \! \delta L \! - \! \lambda_0\, \iota' \iota)^{-1}
   \right) f
\end{eqnarray*}
and
\begin{eqnarray*}
\lefteqn{
   \| \left( (L_0 + \delta L - \lambda_k\, \iota' \iota)^{-1}
           - (L_0 + \delta L - \lambda_0\, \iota' \iota)^{-1}
      \right) f
   \|_{H_{\SL}^- (\cD)}
}
\\
 & \leq &
   |\lambda_k \! - \! \lambda_0|\,
   \| (L_0 \! + \! \delta L \! - \! \lambda_k\, \iota' \iota)^{-1} \|\,
   \| (L_0 \! + \! \delta L \! - \! \lambda_0\, \iota' \iota)^{-1} \|\,
   \| f \|_{H_{\SL}^- (\cD)},
\end{eqnarray*}
estimate (\ref{eq.bound.c11}) implies that in this case the sequence
   $\{ (L_0 + \delta L - \lambda_k\, \iota' \iota)^{-1} f_k \}$ converges to
   $(L_0 + \delta L - \lambda_0\, \iota' \iota)^{-1} f$,
and so
$$
   \left( \Id + C (L_0 + \delta L - \lambda_0\, \iota' \iota)^{-1} \right) f = 0
$$
because of (\ref{eq.ce1}).
But $\lambda_0$ satisfies (\ref{eq.rays.C2}) and
                          $|\lambda| \geq M_\varepsilon$,
and hence the injectivity of the operator
   $\Id + C (L_0 + \delta L - \lambda_0\, \iota' \iota)^{-1}$
established above yields $f = 0$.
This contradicts $\| f \| = 1$.

If $\{ \lambda_k \}$ is unbounded in $\C$ we can repeat the arguments above.
Indeed, then
   $f_k = (L_0 \! + \! \delta L \! - \! \lambda_k\, \iota' \iota) u_k$
and on passing to the weak limit with respect to $k \to \infty$ we get
$$
   f = (L_0 + \delta L) u
     - \lim_{k \to \infty} \lambda_k\, \iota' \iota\, u_k.
$$
As the operator $\iota' \iota$ is compact, the sequence $\{ \iota' \iota\, u_k \}$ converges to
$\iota' \iota\, u$ in the space $H^-_{\SL} (\cD)$.
Moreover, $\iota' \iota\, u \neq 0$ because of (\ref{eq.w01}) and the injectivity of $\iota' \iota$.
This shows that the weak limit
$$
   \lim_{n \to \infty} \lambda_k\, \iota' \iota\, u_k
 = (L_0 + \delta L) u - f
$$
fails to exist if $\{ \lambda_k \}$ is unbounded in $\C$, a contradiction.
Therefore, $N_\varepsilon > 0$ and for all $\lambda \in \C$ satisfying (\ref{eq.rays.C2}) and
                                                                       $|\lambda| \geq M_\varepsilon$
we obtain
\begin{equation}
\label{eq.rays.T71}
   \| \left( \Id + C (L_0 + \delta L - \lambda\, \iota' \iota)^{-1} \right)^{-1}
   \|_{\mathcal{L} (H^-_{\SL} (\cD))}
 \leq
   1 / N_\varepsilon.
\end{equation}

From estimates (\ref{eq.bound.c11}),
               (\ref{eq.rays.T71}) and formula
               (\ref{eq.rays.T8})
it follows that, given any $\lambda \in \C$ satisfying (\ref{eq.rays.C2}) and
                                                       $|\lambda| \geq M_\varepsilon$,
the resolvent $\mathcal{R} (\lambda; T)$ exists and
$$
   \| \mathcal{R} (\lambda; T) \| _{\mathcal{L} (H^-_{\SL} (\cD))}
 \leq
   \mathrm{const}\, (\varepsilon)\, |\lambda|^{-1}.
$$

As $C$ is compact, there are only finitely many $\lambda \in \C$ with
   $|\lambda| < N_\varepsilon$,
such that the operator $(\Id + C (L_0 - \lambda\, \iota' \iota))$ is not injective.
Therefore, it follows from formula (\ref{eq.rays.T8}) that all eigenvalues of the operator
$T$ corresponding to $L_0 \! + \! \delta L \! + \! C$ (except for a finite number) belong
to the corner
$
| \arg \lambda | < \arcsin \| \delta L \| + \varepsilon.
$
Finally, since $\varepsilon > 0$ is arbitrary, all rays (\ref{eq.rays.C1}) are rays of
minimal growth.
By the hypothesis of the theorem, the angles between the pairs of neighbouring rays
   $\arg \lambda = \vartheta$
satisfying (\ref{eq.rays.C1}) are less than $ 2\pi s / n$, and so the statement of the
theorem follows from Theorem \ref{p.root.Schatten}.
\end{proof}

\section{A non-coercive example}
\label{s.4a}

To the best of our knowledge the completeness of root functions has been studied for elliptic
boundary value problems, i.e., for those satisfying the Shapiro-Lopatinskii condition.
In this section we consider an example where the Shapiro-Lopatinskii condition is violated.

Let the complex structure in $\C^n \equiv \mathbb{R}^{2n}$ be given by
   $z_j = x_j + \sqrt{-1} x_{n+j}$
with $j = 1, \ldots, n$ and $\bar{\partial}$ stand for the Cauchy-Riemann operator corresponding
to this structure, i.e., the column of $n$ complex derivatives
$$
   \frac{\partial}{\partial \bar{z}_j}
 = \frac{1}{2}
   \Big( \frac{\partial}{\partial x_j} + \sqrt{-1} \frac{\partial}{\partial x_{n+j}}
   \Big).
$$

The formal adjoint $\bar{\partial}^\ast$ of
                   $\bar{\partial}$
with respect to the usual Hermitian structure in the space $L^2 (\C^n)$ is the line of $n$ operators
$$
 - \frac{1}{2}
   \Big( \frac{\partial}{\partial x_j} - \sqrt{-1} \frac{\partial}{\partial x_{n+j}}
   \Big)
 =:
 - \frac{\partial}{\partial z_j}.
$$
An easy computation shows that $\bar{\partial}^\ast \bar{\partial}$ just amounts to the $- 1/4$
multiple of the Laplace operator
$$
   \iD_{2n}
 = \sum_{j=1}^{2n} \left( \frac{\partial}{\partial x_j} \right)^2
$$
in $\mathbb{R}^{2n}$.

As $A$ we take
$$
   A = - \iD_{2n} + \sum_{j=1}^n a_j (z) \frac{\partial}{\partial \bar{z}_j} + a_0 (z),
$$
where
   $a_1 (z), \ldots, a_n (z)$ and
   $a_0 (z)$
are bounded functions in $\cD$,
   $\cD$ being a bounded domain with Lipschitz boundary in $\C^n$.
The matrix
$$
   \left( a_{i,j} \right)_{\substack{i = 1, \ldots, 2n \\
                                     j = 1, \ldots, 2n}}
$$
is chosen to be
$$
   \Big( \begin{array}{rr}
                     E_n & \sqrt{-1} E_n
\\
         - \sqrt{-1} E_n & E_n
         \end{array}
   \Big)
$$
where $E_n$ is the unity $(n \times n)\,$-matrix.
Obviously, $\overline{a_{i,j}} = a_{j,i}$ for all $i, j = 1, \ldots, 2n$ and the corresponding conormal
derivative is
$$
   \partial_c
 =
   2\, \sum_{j=1}^n (\nu_j \! - \! \sqrt{-1} \nu_{n+j})\, \frac{\partial}{\partial \bar{z}_j}
 =
   \frac{\partial}{\partial \nu}
 + \sqrt{-1}\,
   \sum_{j=1}^n
   \Big( \nu_j \frac{\partial}{\partial x_{n+j}} \! - \! \nu_{n+j} \frac{\partial}{\partial x_{j}} \Big),
$$
which is known as complex normal derivative at the boundary of $\cD$.

Consider the following boundary value problem.
Given a distribution $f$ in $\cD$, find a distribution $u$ in $\cD$ satisfying
\begin{equation}
\label{pr.SL.CR}
\left\{ \begin{array}{rclcl}
   \displaystyle
   \Big( - \iD_{2n} + \sum_{j=1}^n a_j (z) \frac{\partial}{\partial \bar{z}_j} + a_0 (z) \Big) u
 & =
 & f
 & \mbox{in}
 & \cD,
\\
   \partial_c u + b_0 (z) u
 & =
 & 0
 & \mbox{on}
 & \partial \cD.
\end{array}
\right.
\end{equation}

In this case $S$ is empty and
             $b_1 = 1$.
Set $a_{0,0} (z) \equiv 1$ in $\cD$ and
    $b_{0,0}$ to be any positive constant function on the boundary.
Then the corresponding Hermitian form $(\cdot, \cdot)_{\SL}$ is
$$
   (u,v)_{\SL}
 =
   (u,v)_{L^2 (\cD)}
 + (2 \bar{\partial} u,2 \bar{\partial} v)_{L^2 (\cD)}
 + b_{0,0}\, (u,v)_{L^2 (\partial \cD)}
$$
and the space $H_{\SL} (\cD)$ is defined to be the completion of $H^1 (\cD)$ with respect
to the norm
$
   \| u \|_{\SL} = \sqrt{(u,u)_{\SL}}.
$

\begin{lem}
\label{l.SL.CR}
The inclusion
$$
   \iota : H_{\SL} (\cD) \to L^2 (\cD)
$$
is continuous and compact.
More exactly, there are continuous embeddings
\begin{equation}
\label{eq.noncoerc.emb1}
\begin{array}{rcccl}
   H^1 (\cD)
 & \hookrightarrow
 & H_{\SL} (\cD)
 & \hookrightarrow
 & H^{1/2} (\cD),
\\
   H^{-1/2} (\cD)
 & \hookrightarrow
 & H_{\SL}^- (\cD)
 & \hookrightarrow
 & H^{-1} (\cD).
\end{array}
\end{equation}
\end{lem}

\begin{proof}
The lemma is a consequence of results on the Dirichlet problem for the Laplace operator
in the scale of spaces $H^s (\cD)$, where $s \in \mathbb{R}$.

For strongly pseudoconvex domains $\cD \subset \C^n$ the proof can be given within the
framework of complex analysis.
Indeed, it follows from \cite{Stra84} that
$$
   \| u \|_{L^2 (\partial \cD)} \geq c\, \| u \|_{H^{1/2} (\cD)}
$$
for all harmonic functions in $\cD$ of class $H^{1/2} (\cD)$, with $c$ a constant independent
of $u$.
If $\cD$ is pseudoconvex then there is a constant $c > 0$, such that
$$
   \| \bar{\partial} u \|_{L^2 (\cD)} \geq c\, \| u \|_{H^{1/2} (\cD)}
$$
for all $u \in L^2 (\cD)$ orthogonal to the subspace of $L^2 (\cD)$ consisting of holomorphic
functions in $\cD$ (see \cite{Kohn79}).
Since any holomorphic function is harmonic we conclude readily that there are continuous
embeddings of the first line in (\ref{eq.noncoerc.emb1}).
The second line of embeddings follows by familiar duality arguments, thus showing the lemma.
\end{proof}

\begin{rem}
\label{r.regularity.CR}
There is no continuous embedding
$
   H_{\SL} (\cD) \hookrightarrow H^{1/2+\varepsilon} (\cD)
$
with any $\varepsilon > 0$.
This means that the Shapiro-Lopatinskii condition fails to hold for problem (\ref{pr.SL.CR}).
\end{rem}

As $t = 0$ and
   $b_{1} \equiv 1$,
we deduce that estimate (\ref{eq.positive.part1}) is valid, provided that the functions $a_j$ and
                                                                                        $a_0$
are of class $L^\infty (\cD)$ and $b_0 \in L^\infty (\partial \cD)$.
The operator $L_{0}$ corresponds to
$$
\left\{ \begin{array}{rclcl}
   - \iD_{2n} u + u
 & =
 & f
 & \mbox{in}
 & \cD,
\\
   \partial_c u + b_{0,0}\, u
 & =
 & 0
 & \mbox{on}
 & \partial \cD.
\end{array}
\right.
$$

\begin{thm}
\label{p.L0.CR.selfadj}
The inverse $L^{-1}_0$ of the operator $L_0$ induces compact positive self-adjoint operators
$$
\begin{array}{rclcrcl}
   Q_{1}
 & =
 & \iota' \iota\, L^{-1}_{0}
 & :
 & H^-_{\SL} (\cD)
 & \to
 & H^-_{\SL} (\cD),
\\
   Q_{2}
 & =
 & \iota\, L^{-1}_{0}\, \iota'
 & :
 & L^2 (\cD)
 & \to
 & L^2 (\cD),
\\
   Q_{3}
 & =
 & L^{-1}_{0}\, \iota' \iota
 & :
 & H_{\SL} (\cD)
 & \to
 & H_{\SL} (\cD)
\end{array}
$$
which have the same systems of eigenvalues and eigenvectors.
Moreover, all eigenvalues are positive and there are orthonormal bases in
   $H_{\SL} (\cD)$,
   $L^2 (\cD)$ and
   $H^-_{\SL} (\cD)$
consisting of the eigenvectors.
\end{thm}

\begin{proof}
For the proof it suffices to combine Lemmas \ref{l.L0.selfadj} and
                                            \ref{l.SL.CR}.
\end{proof}

Corollary \ref{c.order.SL0} actually shows that the operators $Q_{1}$,
                                                              $Q_{2}$ and
                                                              $Q_{3}$
are of Schatten class $\gS_{n + \varepsilon}$ for any $\varepsilon > 0$, and  so their orders
are finite.

\begin{thm}
\label{t.root.func.CR}
Let $t = 0$ and
    $\delta b_0 = 0$.
Then operator
$
   L : H_{\SL} (\cD) \to H^-_{\SL} (\cD)
$
related to problem (\ref{pr.SL.CR}) is Fredholm.
Moreover, the system of root functions of the corresponding closed operator $T$ is complete
in the spaces $H^-_{\SL} (\cD)$,
              $L^2 (\cD)$ and
              $H_{\SL} (\cD)$.
For any $\varepsilon > 0$, all eigenvalues of $T$ (except for a finite number) belong to the corner
   $|\arg \lambda| < \varepsilon$.
\end{thm}

\begin{proof}
Indeed, the operator
$$
   \delta L = L - L_0 :\, H_{\SL} (\cD) \to H^-_{\SL} (\cD)
$$
maps $H_{\SL} (\cD)$ continuously to
     $L^2 (\cD)$,
and hence it is compact.
Thus, the statement follows from Theorem \ref{t.root.func.1}.
\end{proof}

\begin{cor}
\label{c.root.func.CR}
Let $t = 0$ and
    $\delta b_0 = 0$.
Assume that the constant $c$ of estimate (\ref{eq.positive.part1}) is less than $1$.
Then the operator
$
   L : H_{\SL} (\cD) \to H^-_{\SL} (\cD)
$
corresponding to problem (\ref{pr.SL.CR}) is continuously invertible.
Moreover, the system of root functions of the compact operator $\iota' \iota\, L^{-1}$ in
$H^-_{\SL} (\cD)$ is complete in the spaces
   $H^-_{\SL} (\cD)$,
   $L^2 (\cD)$ and
   $H_{\SL} (\cD)$.
For any $\varepsilon > 0$, all eigenvalues of $\iota' \iota\, L^{-1}$ (except for a finite
number) belong to the corner $|\arg \lambda| < \varepsilon$.
\end{cor}

Recall that estimate (\ref{eq.positive.part1}) in the particular case under considerations
becomes explicitly
$$
   \Big|
   \Big( \sum_{j=1}^n a_j (z) \frac{\partial u}{\partial \bar{z}_j} + (a_0-1)\, u, v
   \Big)_{L^2 (\cD)} \Big|
 \leq
   c\, \| u \|_{\SL} \| v \|_{\SL}
$$
for all $u, v \in  H^1 (\cD)$, with $c$ a constant independent of $u$ and $v$.

\begin{cor}
\label{c.root.func.CR2}
If
   $t = 0$
and
   the constant $c$ of estimate (\ref{eq.positive.part1}) is less than $|\sin (\pi / 2n)|$, %$1$,
then the closed operator $T$ in $H^-_{\SL} (\cD)$ corresponding to problem (\ref{pr.SL.CR}) is continuously invertible.
Moreover, the system of root functions of $T$ is complete in the spaces $H^-_{\SL} (\cD)$,
                                                                        $L^2 (\cD)$ and
                                                                        $H_{\SL} (\cD)$,
and all eigenvalues of $T$ (except for a finite number) lie in the corner
   $|\arg \lambda| < \arcsin c$.
\end{cor}

\section{The coercive case}
\label{s.4b}

We now turn to the coercive case, i.e., we assume that estimate (\ref{eq.coercive}) is fulfilled.
This is obviously the case if all the coefficients $a_{i,j} (z)$ of $A$ are real-valued, which
is due to (\ref{eq.ell}).

\begin{lem}
\label{l.SL.coercive.strong}
Suppose estimate (\ref{eq.coercive}) is fulfilled.
Then there are continuous embeddings
$$
\begin{array}{rcl}
   H_{\SL} (\cD)
 & \hookrightarrow
 & H^{1} (\cD,S),
\\
   H^{-1} (\cD)
 & \hookrightarrow
 & H_{\SL}^- (\cD)
\end{array}
$$
if at least one of the following conditions holds:

\bigskip

1) $S$ is not empty.

2) $\displaystyle \int_{\cD} a_{0,0} (x) dx  > 0$.

3) $\displaystyle \int_{\partial \cD \setminus S} \frac{b_{0,0} (x)}{b_1 (x)}\,  ds > 0$.
\end{lem}

In particular, in this case inclusion (\ref{eq.incl}) is compact.

\begin{proof}
The proof is standard, cf. for instance \cite[Ch.~3, \S~5.6]{Mikh76}.
First we prove that there is a constant $c > 0$, such that
$$
   \| u \|_{H^1 (\cD)}
 \leq
   c\, \| u \|_{\SL}
$$
for all $u \in H^1 (\cD,S)$.
We argue by contradiction.
Let for any natural number $k$ there be a function $u_k$ in $H^1 (\cD,S)$ satisfying
$
   \| u_k \|_{H^1 (\cD)} > k\, \| u_k \|_{\SL}.
$
Then $u_k \neq 0$, and so replacing $u_k$ by $u_k / \| u_k \|_{H^1 (\cD)}$, if necessary,
we can assume that
   $\| u_k \|_{H^1 (\cD)} = 1$
and
$
   \| u_k \|_{\SL} < 1 / k
$
for all $k \in \mathbb{N}$.
Thus, the sequence $\{ u_k \}$ converges to zero in $H_{\SL} (\cD)$.
In particular, for each derivative $\{ \partial_j u_k \}$ with $1 \leq j \leq n$, we have
\begin{equation}
\label{eq.1}
   \lim_{k \to \infty} \| \partial_j u_k \|_{L^2 (\cD)} = 0
\end{equation}
because of estimate (\ref{eq.coercive}).

On the other hand, the weak compactness principle shows that on passing to a subsequence,
if necessary, we may assume that the sequence $\{ u_k \}$ converges weakly to an element
$u$ in $H^1 (\cD)$.
As the embedding
   $H^{1} (\cD) \hookrightarrow L^2 (\cD)$
is compact, the sequence $\{ u_k \}$ converges to $u$ in $L^2 (\cD)$.
Then it follows from (\ref{eq.1}) that $\{ u_k \}$ converges to $u$ in  $H^1 (\cD)$ whence
\begin{eqnarray*}
   \lim_{k \to \infty} \| u_k \|_{H^1 (\cD)}
 & = &
   \| u \|_{H^1 (\cD)}
\\
 & = &
   1
\end{eqnarray*}
and the derivatives of the limit function $u$ vanish almost everywhere in $\cD$,
   i.e.,
\begin{equation}
\label{eq.2}
   |u| = \frac{1}{\displaystyle \int_{\cD} dx} > 0.
\end{equation}

If $S$ is not empty, then the traces of $u_k$ on $S$ converge, by the trace theorem, to the
trace of $u$ on $S$.
Since $u_k = 0$ on $S$ for all $k$, we get $u = 0$ which contradicts (\ref{eq.2}).

Furthermore, as $\displaystyle \lim_{k \to \infty} \| u_k \|_{\SL} = 0$, we see that
\begin{eqnarray*}
   |u| \int_{\cD} a_{0,0} (x) dx
 & =
 & 0,
\\
   |u| \int_{\partial \cD \setminus S} \frac{b_{0,0} (x)}{b_1 (x)}\, ds
 & =
 & 0,
\end{eqnarray*}
i.e., (\ref{eq.2}) contradicts either of the conditions 2) and 3).
Therefore, the space $H_{\SL} (\cD)$ is continuously embedded to the space $H^{1} (\cD,S)$.
The second embedding follows by duality.
\end{proof}

\begin{lem}
\label{l.SL.coercive.equiv}
Assume that estimate (\ref{eq.coercive}) is fulfilled.
If
   one of the three conditions of Lemma \ref{l.SL.coercive.equiv} holds true
and
   $b_{0,0} / b_1 \in L^\infty (\partial \cD \setminus S)$,
then the norms $\| \cdot \|_{\SL}$ and
               $\| \cdot \|_{H^1 (\cD)}$
are equivalent and the spaces $H_{\SL} (\cD)$ and
                              $H^1 (\cD,S)$
coincide as topological ones.
\end{lem}

\begin{proof}
We have already proved that, under estimate (\ref{eq.coercive}) and
                             one of the three conditions of Lemma \ref{l.SL.coercive.equiv},
the norm $\| \cdot \|_{\SL}$ is not weaker than the norm
         $\| \cdot \|_{H^1 (\cD)}$
on $H^1 (\cD,S)$.
Furthermore, since
   $a_{i,j}$ and $a_{0,0}$ are bounded functions in $\cD$ and
   $b_{0,0} / b_1$ a bounded function on $\partial \cD \setminus S$,
we obtain
\begin{eqnarray*}
   \int_{\cD}
   \sum_{i,j = 1}^n a_{i,j} (x)\,  \overline {\partial_i u (x)} \partial_j u (x)\,
   dx
 & \leq &
   c\, \sum_{j=1}^n  \| \partial_j u \|^2_{L^2 (\cD)},
\\
   \int_{\cD}
   a_{0,0} (x)\, |u (x)|^2\, dx
 & \leq &
   c\, \| u \|^2_{L^2 (\cD)},
\\
   \int_{\partial \cD \setminus S}
   \frac{b_{0,0} (x)}{b_1 (x)}\, |u (x)|^2\, ds
 & \leq &
   c\, \| u \|^2_{H^1 (\cD)}
\end{eqnarray*}
for all $u \in H^1 (\cD,S)$,
   the last inequality being a consequence of the trace theorem for Sobolev spaces.
Here, by $c$ is meant a constant independent of $u$, which can be diverse in different
applications.

On combining the above estimates we deduce immediately that there is a constant $c$ with
the property that
$$
   \| u \|_{\SL} \leq c\, \| u \|_{H^1 (\cD)}
$$
for all $u \in H^1 (\cD,S)$, as desired.
\end{proof}

Now we want to understand what perturbations of $a_j$ and
                                                $\delta b_{0} / b_1$,
                                                $t$
preserve the completeness property of root functions of the operator $L_0^{-1}$.

\begin{lem}
\label{l.Fredholm}
Let estimate
   (\ref{eq.coercive})
and
   one of the three conditions of Lemma \ref{l.SL.coercive.strong}
be fulfilled.
If
\begin{equation}
\label{eq.positive.part2}
   \Big| \int_{\partial \cD \setminus S} b_1^{-1} \partial_t u\, \overline{v}\, ds \Big|
 \leq c\, \| u \|_{\SL} \| v \|_{\SL}
\end{equation}
for all $u, v \in H^1 (\cD,S)$, with a constant $c < 1$ independent on $u$ and $v$, and there is
$s < 1/2$ such that
\begin{equation}
\label{eq.3}
   |((b_1^{-1} \delta b_0) u, v)_{L^2 (\partial \cD \setminus S)}|
 \leq
   c\, \| u \|_{H^s (\partial \cD)}\, \| v \|_{\SL}
\end{equation}
for all $u, v \in H^1 (\cD,S)$, then inequality (\ref{eq.positive.part1}) holds and
                                        problem (\ref{eq.SL.w}) is Fredholm of index zero.
\end{lem}

\begin{proof}
Denote by $\delta L : H_{\SL} (\cD) \to H^-_{\SL} (\cD)$ the operator defined by $\partial_t$,
as is described in  Section \ref{s.4}.
By (\ref{eq.positive.part2}),
$$
   \| \delta L \|_{\mathcal{L} (H_{\SL} (\cD),H^-_{\SL} (\cD))}
 < 1
 = \| L_0^{-1} \|_{\mathcal{L} (H^-_{\SL} (\cD),H_{\SL} (\cD))},
$$
and so a familiar argument shows that the operator
   $L_0 + \delta L : H_{\SL} (\cD) \to H^-_{\SL} (\cD)$
is invertible.

The differential operator
$$
   \delta A = \sum_{j=1}^n a_j (x)\, \partial_j + \delta a_0\, (x)
$$
maps the space $H^1 (\cD)$ continuously to
               $L^2 (\cD)$.
Lemma \ref{l.SL.coercive.strong} implies readily that $\delta A$ maps also the space
$H_{\SL} (\cD)$ continuously to $L^2 (\cD)$.
In particular, we get
\begin{equation}
\label{eq.positive.part11}
   |(\delta A u, v)_{L^2 (\cD)}|
 \leq
   \mathrm{const}\, \| u \|_{H^{1} (\cD)} \| v \|_{L^2 (\cD)}
 \leq
   \mathrm{const}\, \| u \|_{\SL} \| v \|_{\SL}
\end{equation}
for all $u, v \in H_{\SL} (\cD)$.
As the embedding
   $\iota' : L^2 (\cD) \hookrightarrow H^-_{\SL} (\cD)$
is compact (see Lemma \ref{l.dual.emb}), the operator $\delta A$ induces a compact operator
$$
   C_1 :  H_{\SL} (\cD) \to H^-_{\SL} (\cD).
$$

By (\ref{eq.3}),
$$
   |((b_1^{-1} \delta b_0) u, v)_{L^2 (\partial \cD \setminus S)}|
 \leq
   \mathrm{const}\, \| u \|_{H^s (\partial \cD)} \| v \|_{\SL}
 \leq
   \mathrm{const}\, \| u \|_{\SL} \| v \|_{\SL}
$$
for all $u, v \in H^1 (\cD,S)$.
Combining this estimate with (\ref{eq.positive.part11}) and
                             (\ref{eq.positive.part2})
shows that inequality (\ref{eq.positive.part1}) is fulfilled with $c$ a constant independent of
$u$ and
$v$.
Moreover, the summand
   $((b_1^{-1} \delta b_0) u, v)_{L^2 (\partial \cD \setminus S)}$
in the weak formulation (\ref{eq.SL.w}) of problem (\ref{eq.SL}) induces a bounded linear
operator
$$
   C_2 :  H_{\SL} (\cD) \to H^-_{\SL} (\cD)
$$
satisfying
\begin{equation}
\label{eq.est.comp}
   \| C_2 u \|_{H_{\SL}^- (\cD)}
 \leq
   \mathrm{const}\, \| u \|_{H^s (\partial \cD)}.
\end{equation}
Now, if $\sigma$ is a bounded set in $H_{\SL} (\cD)$ then,
   by Lemma \ref{l.SL.coercive.strong},
it is also bounded in $H^1 (\cD)$ and,
   by the trace theorem,
the set $\sigma \restriction_{\partial \cD}$ of restrictions of elements of $\sigma$ to
$\partial \cD$ is bounded in $H^{1/2} (\partial \cD)$, too.
The Rellich theorem implies that
   $\sigma \restriction_{\partial \cD}$
is precompact in $H^s (\partial \cD)$, if $s < 1/2$.
Hence, for any sequence $\{ u_k \}$ in $\sigma$ there is a subsequence $\{ u_{k_j} \}$ whose
restriction to $\partial \cD$ is a Cauchy sequence in $H^s (\partial \cD)$.
Using (\ref{eq.est.comp}) we conclude that any sequence $\{ u_k \}$ in $\sigma$ has a subsequence
$\{ u_{k_j} \}$, such that $\{ C_2 u_{k_j}\}$ is a Cauchy sequence in $H^-_{\SL} (\cD)$.
Thus, the operator $C_2 : H_{\SL} (\cD) \to H^-_{\SL} (\cD)$ is compact.

Finally, by construction, the operator
$$
   L - (L_0 \! + \! \delta L) = C_1 \! + \! C_2 :\, H_{\SL} (\cD) \to H^-_{\SL} (\cD)
$$
is compact.
Therefore, problem (\ref{eq.SL.w}) is equivalent to the Fredholm-type operator equation
$$
   (\Id + (L_0 \! + \! \delta L)^{-1} (C_1 \! + \! C_2)) u = (L_0 \! + \! \delta L)^{-1} f
$$
in $H_{\SL} (\cD)$ with compact operator
   $(L_0 \! + \! \delta L)^{-1} (C_1 \! + \! C_2) : H_{\SL} (\cD) \to H_{\SL} (\cD)$.
This establishes the lemma.
\end{proof}

In fact condition (\ref{eq.3}) means that the multiplication operator by $b_1^{-1} \delta b_0$
maps the traces of elements of $H_{\SL} (\cD)$ on the boundary continuously to
   $H^{-s} (\partial \cD \setminus S)$
with some $s < 1/2 $.

\begin{lem}
\label{l.eq.3}
If  $b_1^{-1} \delta b_0 \in L^\infty (\partial \cD \setminus S)$ then
$$
   |((b_1^{-1} \delta b_0) u, v)_{L^2 (\partial \cD \setminus S)}|
 \leq
   c\, \| u \|_{L^2 (\partial \cD)}\, \| v \|_{\SL}
$$
for all $u, v \in H^1 (\cD,S)$, the constant $c$ being independent of $u$ and
                                                                      $v$.
\end{lem}

\begin{proof}
Since $b_1^{-1} \delta b_0$ is a bounded function on $\partial \cD \setminus S$, it follows
that
\begin{eqnarray*}
   |((b_1^{-1} \delta b_0) u, v)_{L^2 (\partial \cD \setminus S)}|
 & \leq &
   c\, \| u \|_{L^2 (\partial \cD)} \| v \|_{L^2 (\partial \cD)}
\\
 & \leq &
   c\, \| u \|_{L^2 (\partial \cD)} \| v \|_{H^1 (\cD)}
\\
 & \leq &
   c_1\, \| u \|_{L^2 (\partial \cD)} \| v \|_{\SL}
\end{eqnarray*}
for all $u, v \in H^1 (\cD,S)$, as desired.
\end{proof}

Let
   $t_1 (x), \ldots, t_{n-1} (x)$
be a basis of tangential vectors of the boundary surface at a point $x \in \partial \cD$.
Then we can write
$$
   \partial_t u
 = \sum_{j=1}^{n-1} c_j (x) \, \partial_{t_j}
$$
with  $c_1, \ldots, c_{n-1}$ bounded functions on the boundary vanishing on $S$.
As $\partial_{t_j}$ are first order differential operators with bounded coefficients they map
   $H^{1/2} (\partial \cD)$ continuously to
   $H^{-1/2} (\partial \cD)$.
Given any $u, v \in H^1 (\cD,S)$, we get
\begin{eqnarray*}
   \Big| \int_{\partial \cD \setminus S} b_1^{-1} \partial_t u\, \overline{v}\, ds
   \Big|
 & \leq &
   c\,
   \sum_{j=1}^{n-1} \| \partial_{t_j} u \|_{H^{-1/2} (\partial \cD)}
                    \| v \|_{H^{1/2} (\partial \cD)}
\\
 & \leq &
   c\, \| u \|_{H^{1/2} (\partial \cD)} \| v \|_{H^{1/2} (\partial \cD)}
\\
 & \leq &
   c\, \| u \|_{H^{1} (\cD)} \|v\|_{H^{1} (\cD)}
\\
 & \leq &
   c\, \| u \|_{\SL} \| v \|_{\SL},
\end{eqnarray*}
provided that
   $b_1^{-1} c_j \in L^\infty (\partial \cD)$ for $1 \leq j \leq n-1$,
the constant $c$ does not depend on $u$ and
                                    $v$
and may be diverse in different applications.

\begin{thm}
\label{p.Fredholm1}
Let estimates (\ref{eq.coercive}),
               (\ref{eq.positive.part1})
and one of the three conditions of Lemma \ref{l.SL.coercive.strong} be fulfilled.
If the constant $c$ of inequality (\ref{eq.positive.part1}) is less than $1$, then problem (\ref{eq.SL.w}) is uniquely solvable and the inverse operator
   $L^{-1} : H_{\SL}^- (\cD) \to H_{\SL} (\cD) \subset H^1 (\cD)$
is bounded.
\end{thm}

\begin{proof}
This follows from Lemmas \ref{l.solv.SL2} and \ref{l.SL.coercive.strong}.
\end{proof}

Under the hypotheses of Theorem \ref{p.Fredholm1}, if moreover $t = 0$, then the system of root
functions of the compact operator
$
   \iota' \iota\, L^{-1}:  H^-_{\SL} (\cD) \to H^-_{\SL} (\cD)
$
is complete in the spaces $H^-_{\SL} (\cD)$,
                          $L^2 (\cD)$ and
                          $H_{\SL} (\cD) \subset H^1 (\cD)$.
Moreover, for each $\varepsilon > 0$, all eigenvalues of $\iota' \iota\, L^{-1}$ (except possibly
for a finite number) belong to the corner $|\arg \lambda| < \varepsilon$.

\begin{cor}
\label{c.root.func.4}
If the constant $c$ of inequality (\ref{eq.positive.part1}) is less than $|\sin (\pi /n)|$ then
the operator
$
   T : H_{\SL} (\cD) \to H^-_{\SL} (\cD)
$
corresponding to problem (\ref{eq.SL.w}) is continuously invertible.
Moreover, the system of root functions of the operator $T$ is complete in the spaces
   $H^-_{\SL} (\cD)$,
   $L^2 (\cD)$ and $H_{\SL} (\cD)$,
and all eigenvalues of $T$ belong to the corner $|\arg \lambda| \leq \arcsin c$.
\end{cor}

\begin{proof}
The statement follows from Theorem \ref{t.root.func.Schatten} and
                           Lemmas \ref{l.SL.coercive.strong} and \ref{l.Fredholm}.
\end{proof}

\begin{thm}
\label{t.root.func.coercive.2}
Suppose that
   estimate (\ref{eq.coercive}) and
   one of the three conditions of Lemma \ref{l.SL.coercive.strong}
are fulfilled.
If
   $t = 0$ and
   $b_1^{-1} \delta b_0 \in L^\infty (\partial \cD \setminus S)$,
then the operator
$
   L : H_{\SL} (\cD) \to H^-_{\SL} (\cD)
$
related to Problem (\ref{eq.SL.w}) is Fredholm.
Moreover, the system of root functions of the corresponding closed operator
$
   T : H_{\SL}^- (\cD) \to H_{\SL}^- (\cD)
$
is complete in the spaces $H^-_{\SL} (\cD)$,
                          $L^2 (\cD)$ and
                          $H^1 (\cD,S)$,
and, for arbitrary $\varepsilon > 0$, all eigenvalues of $T$ (except for a finite number)
belong to the corner $|\arg \lambda| < \varepsilon$ in $\C$.
\end{thm}

\begin{proof}
For the proof it suffices to apply Theorem \ref{p.root.func.1} combined with
                                   Lemmas \ref{l.SL.coercive.strong} and \ref{l.Fredholm}.
\end{proof}

\begin{thm}
\label{t.root.func.6}
Let
   estimate (\ref{eq.coercive}) and
   one of the three conditions of Lemma \ref{l.SL.coercive.strong}
be fulfilled.
If
   $b_1^{-1} \delta b_0 \in L^\infty (\partial \cD \setminus S)$
and
   the constant $c$ of estimate (\ref{eq.positive.part2}) is less than $|\sin \pi / n|$,
then the operator
$
   L : H_{\SL} (\cD) \to H^-_{\SL} (\cD)
$
related to problem (\ref{eq.SL.w}) is Fredholm.
Moreover, the system of root functions of the corresponding closed operator $T$ in $H^-_{\SL} (\cD)$
is complete in the spaces $H^-_{\SL} (\cD)$,
                          $L^2 (\cD)$ and
                          $H^1 (\cD,S)$,
and, for any $\varepsilon > 0$, all eigenvalues of $T$ (except for a finite number) lie in the corner
$
|\arg \lambda| < \arcsin c + \varepsilon.
$
\end{thm}

\begin{proof}
This is a consequence of Theorem \ref{t.root.func.Schatten.Comp} and
                         Lemmas \ref{l.SL.coercive.strong} and \ref{l.Fredholm}.
\end{proof}

\section{Zaremba type problems}
\label{s.8}

As but one example of boundary value problems satisfying the Shapiro-Lopatins\-kii condition we
consider the mixed problem
\begin{equation}
\label{eq.Zar}
\left\{
\begin{array}{rclcl}
   \displaystyle
   - \iD_n u  + \sum_{j=1}^n a_j (x) \partial_j u + a_0 (x) u
 & =
 & f
 & \mbox{in}
 & \cD,
\\
   u
 & =
 & 0
 & \mbox{on}
 & S,
\\
   \partial_v u
 & =
 & 0
 & \mbox{on}
 & \partial \cD \setminus S
\end{array}
\right.
\end{equation}
for a real-valued function $u$, where
   $\iD_n$ is the Laplace operator in $\mathbb{R}^n$,
   the coefficients $a_1, \ldots, a_n$ and $a_0$ are assumed to be bounded functions in $\cD$,
and
$
   \partial_v = \partial_\nu + \varepsilon \partial_t
$
with
   a tangential vector field $t (x)$ on $\partial \cD$ whose coefficients are bounded functions
   vanishing on $S$,
and
   $\varepsilon \in \C$
(cf. \cite{Zare10}).

In this case
   $a_{i,j} = \delta_{i,j}$,
   $b_0 = \chi_S$ is the characteristic function of the boundary set $S$, and
   $b_1 = \chi_{\partial \cD \setminus S}$ is that of $\partial \cD \setminus S$.

From results of the previous section it follows that the root functions related to problem
(\ref{eq.Zar}) in the space $H_{\SL} (\cD)$ are complete in $H^-_{\SL} (\cD)$,
                                                            $L^2 (\cD)$ and
                                                            $H^1 (\cD,S)$
for all $\varepsilon$ of sufficiently small modulus.

We finish the paper by showing a second order elliptic differential operator in $\mathbb{R}^2$
for which no Zaremba-type problem is Fredholm.
The idea is traced back to a familiar example of A.V. Bitsadze (1948).

\begin{exmp}
\label{ex.baf}
Let $A = \bar{\partial}^2$ be the square of the Cauchy-Riemann operator in the plane of
complex variable $z$.
We choose $\cD$ to be the upper half-disk of radius $1$, i.e., the set of all $z \in \C$
satisfying $|z| < 1$ and $\Im z > 0$.
As $S$ we take the upper half-circle, i.e., the part of $\partial \cD$ lying in the upper
half-plane.
Consider the function sequence
$$
   u_k (z) = (|z|^2 - 1)\, \frac{\sin (k z)}{k^s},
$$
for $k = 1, 2, \ldots$,
   where $s$ is a fixed positive number.
Each function $u_k$ vanishes on $S$.
Moreover, for any differential operator $B$ of order $< s$ with bounded coefficients, the
sequence $\{ B u_k \}$ converges to zero uniformly on $\partial \cD \setminus S = [-1,1]$.
Since $|u_k (z)| \to \infty$ for all $z \in \cD$, we deduce that no reasonable setting of
Zaremba-type problem is possible.
\end{exmp}

\bigskip

{\sc Acknowledgements\,}
The first author gratefully acknowledges the financial support of the RFBR grant
   11-01-91330-NNIO\_a.

\end{document}